\documentclass[aop,MSNbibl,dvips]{arximspdf}

\doi{10.1214/11-AOP737} 
\volume{41}
\issue{1}
\pubyear{2013}
\firstpage{294}
\lastpage{328}

\makeatletter
\newcommand{\rrvert}{\vert}
\newcommand{\llvert}{\vert}

\newcommand{\Cdu}{\frac{C}{2}}

\newcommand{\implies}{\Longrightarrow}

\newcommand{\hb}{\bar{h}}
\newcommand{\lam}{{\lambda}}

\newcommand{\Var}[0]{\operatorname{Var}}
\newcommand{\Cov}[0]{\operatorname{Cov}}
\newcommand{\E}[0]{{E}}
\newcommand{\Vol}[0]{\operatorname{Vol}}
\newcommand{\SM}[0]{\operatorname{SM}}
\newcommand{\LM}[0]{\operatorname{LM}}

\newtheorem{maintheorem}{Theorem}
\newtheorem{lemma}{Lemma}
\newtheorem{proposition}{Proposition}
\newproclaim{definition}{Definition}
\newtheorem{conjecture}{Conjecture}

\makeatother

\begin{document}
\begin{frontmatter}

\title{Exact thresholds for Ising--Gibbs samplers on general graphs}
\runtitle{Exact thresholds for Ising--Gibbs samplers}

\begin{aug}
\author[A]{\fnms{Elchanan} \snm{Mossel}\corref{}\thanksref{t2}\ead[label=e1]{mossel@stat.berkeley.edu}}
\and
\author[A]{\fnms{Allan} \snm{Sly}\thanksref{t3}\ead[label=e2]{sly@stat.berkeley.edu}}
\runauthor{E. Mossel and A. Sly}
\affiliation{University of California, Berkeley,
and Weizmann Institute, and University of California, Berkeley}
\address[A]{Department of Statistics\\
University of California, Berkeley\\
367 Evans Hall\\
Berkeley, California 94720\\
USA\\
\printead{e1}\\
\hphantom{E-mail: }\printead*{e2}} 
\end{aug}

\thankstext{t2}{Supported by an Alfred Sloan fellowship
in Mathematics, by NSF CAREER Grant DMS-05-48249 (CAREER),
by DOD ONR Grant N0014-07-1-05-06, by BSF Grant 2004105 and by ISF Grant
1300/08.}

\thankstext{t3}{Supported in part by an Alfred Sloan Fellowship in
Mathematics.}

\received{\smonth{1} \syear{2010}}
\revised{\smonth{7} \syear{2011}}

%
\begin{abstract}
We establish tight results for rapid mixing of Gibbs samplers for the
Ferromagnetic Ising model on general graphs. We show that if
\[
(d-1) \tanh\beta< 1,
\]
then there exists a constant $C$ such that the discrete time mixing
time of Gibbs samplers for the ferromagnetic Ising model on \textit{any}
graph of $n$ vertices and maximal degree $d$, where all interactions
are bounded by $\beta$, and arbitrary external fields are bounded by $C
n \log n$. Moreover, the spectral gap is uniformly bounded away from 0
for all such graphs, as well as for infinite graphs of maximal degree
$d$.

We further show that when $d \tanh\beta< 1$, with high probability
over the Erd\H{o}s--R\'enyi random graph $G(n,d/n)$, it holds that the
mixing time of Gibbs samplers is
\[
n^{1+\Theta({1}/{\log\log n})}.
\]
Both results are tight, as it is known that the mixing time for random
regular and Erd\H{o}s--R\'enyi random graphs is, with high probability,
exponential in $n$ when $(d-1) \tanh\beta> 1$, and $d \tanh\beta>
1$, respectively. To our knowledge our results give the first tight
sufficient conditions for rapid mixing of spin systems on general
graphs. Moreover, our results are the first rigorous results
establishing exact thresholds for dynamics on random graphs in terms of
spatial thresholds on trees.
\end{abstract}

%

\end{frontmatter}

\section{Introduction}\label{intro}
Gibbs sampling is a standard model in statistical physics for the
temporal evolution of spin systems as well as
a popular technique for sampling high-dimensional distributions.
The study of the convergence rate of Gibbs samplers has thus attracted
much attention from both statistical physics and theoretical computer
science. Traditionally such systems where studied on lattices. However,
the applications in computer science, coupled with the interest in
diluted spin-glasses in theoretical physics, led to an extensive
exploration of properties of Gibbs sampling on general graphs of
bounded degrees.

Below we will recall various definitions for measuring the convergence
rate of the dynamics in spectral and total variation forms.
In particular, we will use the notion of \textit{rapid mixing} to indicate
convergence in polynomial time in the size of the underlying
graph.

A feature of most sufficient conditions for rapid convergence is that
they either apply to general graphs,
but are not (known to be) tight,
or the results are known to be tight, but apply only to special
families of graphs, like $2$-dimensional grids, or trees.
Examples of results of the first type include the Dobrushin and the
Dobrushin--Shlosman conditions \cite{DobrushinShlosman85} and
results by Vigoda and collaborators on colorings; see, for
example, \cite{Vigoda99,Vigoda00,HayesVigoda05}. Examples of tight
results for special graphs include the Ising model on $2$-dimensional
grids by Martinelli and Oliveri \cite
{MartinelliOliveri94a,MartinelliOliveri94b};
see also \cite{Martinelli99} and the Ising model on trees \cite
{KeMoPe01,BeKeMoPe05,MaSiWe03a,MaSiWe04}.

In this paper, we consider Gibbs sampling for the ferromagnetic Ising
model on general graphs and provide a criteria in terms of
the maximal coupling constant $\beta$ and the maximal degree $d$ which
guarantees rapid convergence for \textit{any} graph and \textit{any} external
fields. The criteria is $(d-1) \tanh\beta< 1$.
We further establish that if $d \tanh\beta< 1$, then rapid mixing
holds, with high probability, on the Erd\H{o}s--R\'enyi random graph
of average degree $d$, thus proving the main conjecture of \cite
{MosselSly08,MosselSly09}.
Both results are tight as random $d$-regular graphs and Erd\H{o}s--R\'
enyi random graph of average degree $d$ with no external fields,
have, with high probability, mixing times that are exponential in the
size of the graph when $(d-1) \tanh\beta> 1$ (resp., $d \tanh\beta>
1$) \cite{GerschenfeldMontanari07,DemboMontanari09}. To our
knowledge, our results are the first tight sufficient conditions for
rapid mixing of spin systems on general graphs.

Our results are intimately related to the spatial mixing properties of
the Gibbs measure, particularly on trees. A model has the
\textit{uniqueness property} (roughly speaking) if the marginal spin at a
vertex is not affected by conditioning the spins of sets of distant
vertices as the distance goes to infinity. On the infinite $d$-regular
tree, uniqueness of the ferromagnetic Ising model holds when
$(d-1)\tanh\beta\leq1$~\cite{Lyons89}, corresponding to the region
of rapid mixing. It is known from the work of Weitz \cite{Weitz06}
that in fact spatial mixing occurs when $(d-1)\tanh\beta\leq1$ on
any graph of maximum degree $d$.

It is widely believed that (some form of) spatial mixing implies fast
mixing of the Gibbs sampler. However, this is only known for amenable
graphs and for a strong form of spatial mixing called ``strong spatial
mixing'' \cite{DSVW04}. While lattices are amenable, there are many
ensembles of graphs which are nonamenable such as expander graphs. In
fact, since most graphs of bounded degree are expanders, the strong
spatial mixing technique does not apply to them.
Our results apply to completely general graphs and in particular
various families of random graphs whose neighborhoods have exponential growth.

Our results also immediately give lower bounds on the spectral gap of
the continuous time Glauber dynamics which are independent of the size
of the graph. This in turn allows us to establish a
lower bound on the spectral gap for the Glauber dynamics on infinite
graphs of maximal degree bounded by $d$, as well.

To understand our result related to the Erd\H{o}s--R\'enyi random
graph, we note that
the threshold for the Erd\H{o}s--R\'enyi random graphs also
corresponds to a spatial mixing threshold. For a randomly chosen
vertex, the local graph neighborhood is asymptotically distributed as a
Galton--Watson branching process with offspring distribution Poisson
with mean $d$. Results of Lyons \cite{Lyons89} imply that the
uniqueness threshold on the Galton--Watson tree is $d\tanh\beta<1$,
which is equal to the threshold for rapid mixing established here.

The correspondence between spatial and temporal mixing is believed to
hold for many other important models. We conjecture that when there is
uniqueness on the $d$-regular tree for the \textit{antiferromagnetic
Ising} model or the \textit{hardcore model}, then there is rapid mixing of
the Gibbs sampler on all graphs of maximum degree $d$ in these models.
It is known that for both these models that the mixing time on almost
all random $d$-regular bipartite graphs is exponential in $n$ the size
of the graph beyond the uniqueness threshold \cite
{MoWeWo09,GerschenfeldMontanari07,DemboMontanari09}, so our
conjecture is that uniqueness on the tree exactly corresponds to rapid
mixing of the Gibbs sampler. We summarize our main contributions as follows:
%
\begin{itemize}
\item
Our results are the first results providing tight criteria for rapid
mixing of Gibbs samplers on general graphs.
\item
Our results show that the threshold is given by a corresponding
threshold for a tree model, in particular, in the case of random graphs
and dilute mean field models. We note that in the theory of
spin-glasses, it is conjectured that for many spin systems on random
diluted (bounded average degree) graphs the ``dynamical threshold'' for
rapid mixing is given by a corresponding ``replica'' threshold, that
is, a spatial threshold for a corresponding spin system on trees; see,
for example, \cite{MezMon09,KMRSZ07,MoRiSe08}. To the best of our
knowledge our results are the first to rigorously establish such thresholds.
\end{itemize}

While the proof we present here is short and elegant, it is
fundamentally different than previous approaches in the area. In
particular:
\begin{itemize}
\item
It is known that imitating the block dynamics technique \cite
{MartinelliOliveri94a,MartinelliOliveri94b} cannot be extended to the
nonamenable setting since the bounds rely crucially on the small
boundary-to-volume ratio which can no be extended to expander graphs;
see a more detailed discussion in \cite{DSVW04}.
\item Weitz \cite{Weitz06} noted that the tree of self avoiding walks
construction establishes mixing results on amenable graphs, but not for
nonamenable graphs. In general, correlation inequalities/spatial
mixing have previously only been shown to to imply rapid mixing on
amenable graphs; an excellent reference is the thesis of Weitz
\cite{Weitz05}.
\item
The technique of censoring the dynamics is another recent development
in the analysis of Gibbs samplers \cite{Weitz06} and can, for
instance, be used to translate results on the block dynamics to those
on the single site dynamics. Its standard application does not,
however, yield new results for nonamenable graphs.
\item
While tight results have been established in the case of trees \cite
{KeMoPe01,BeKeMoPe05,MaSiWe03a,MaSiWe04} which are nonamenable,
the methods do not generalize to more general graphs, as they make
fundamental use of properties of the tree, in particular, the presence
of leaves at the base. Indeed, the fact that the median degree of a
tree is 1 illustrates the difference between trees and regular graphs.
\end{itemize}
The main novelty in our approach is a new application of the censoring
technique. In the standard use of censoring, a censored Markov chain is
constructed which is shown to mix rapidly, and then the censoring
inequality implies rapid mixing of the original dynamics. Our approach
is a subtle conceptual shift. Rather than construct a censoring scheme
which converges to the stationary distribution, we construct a sequence
of censored dynamics which do not converge to stationarity. They do,
however, allow us to establish a sequence of recursive bounds from
which we derive our estimates of the spectral gap and the mixing time.

Another serious technical challenge of the paper was determining the
correct mixing time for the Gibbs sampler on Erd\H{o}s--R\'enyi random
graphs. The necessary estimate is to bound the mixing time on the local
neighborhoods of the graph which are Galton--Watson branching processes
with Poisson offspring distribution. This is done via an involved
distributional recursive analysis of the cutwidth of these branching
process trees.

In the following subsections, we state our results, and then we recall
the definition of the Ising model, Gibbs sampling and Erd\H{o}s--R\'
enyi random graphs. This is followed by a statement of a general
theorem, from which both of our main results follow. We then sketch the
main steps of the proof, which are followed by detailed proofs. We then
show how our spectral gap bounds on finite graphs can be extended to
infinite graphs. Finally we conclude with open problems involving other systems.

\subsection{Our results}

In our main result we establish the following tight criteria for rapid
mixing of Gibbs sampling for general graphs in terms of the maximal degree.

\begin{maintheorem}\label{tmain}
For any integer $d \geq2$, and inverse temperature $\beta> 0$, such that
%
\begin{equation}
\label{emainCondition} (d-1)\tanh\beta< 1,
\end{equation}
there exist constants
$0 < \lambda^*(C,\beta), C(d,\beta) < \infty$, such that
on any graph of maximum degree $d$ on $n$ vertices, the discrete time
mixing time of the Gibbs sampler for the ferromagnetic Ising model with
all edge interactions bounded by $\beta$, and arbitrary external
fields, is bounded above by $C n \log n$.

Further the continuous time spectral gap of the dynamics is bounded
below by~$\lambda^*$. The spectral gap bound applies also for infinite graphs.
\end{maintheorem}

We note that a lower bound of $\Omega(n \log n)$ on the mixing time
follows from the general results of \cite{HayesSinclair05}.

The techniques we develop here also allow us to derive results for
graphs with unbounded degrees. Of particular interest is the following
tight result:
%
\begin{maintheorem}\label{tER}
Let $\beta> 0$ and $d > 0$ and consider the Erd\H{o}s--R\'enyi random
graph $G$ on $n$ vertices, where each edge is present independently
with probability $d/n$. Then for all $\beta$ such that $d \tanh\beta
< 1$, there exists $c(d,\beta)$ and $C(d,\beta)$,
such that with high probability over $G$,
the discrete time mixing time $\tau_{\mathrm{mix}}$ of the Gibbs sampler for
the ferromagnetic Ising model with all edge interactions bounded by
$\beta$ and arbitrary external field satisfies
\[
n^{ (1+{c}/{\log\log n} )} \leq\tau_{\mathrm{mix}} \leq n^{ (1+{C}/{\log\log n} )},
\]
while the continuous time spectral gap satisfies
\[
n^{-{c}/{\log\log n}} \geq\mbox{Gap} \geq n^{-{C}/{\log
\log n}}.
\]
\end{maintheorem}

Both results are tight as estimates obtained in \cite
{GerschenfeldMontanari07,DemboMontanari09},
following \cite{MoWeWo09}, and they prove a conjecture from \cite
{MosselSly08,MosselSly09}, implying that for the Ising model without
external fields, the mixing time of the Gibbs sampler is, with high
probability, $\exp(\Omega(n))$ on random $d$-regular graphs if
$(d-1)\tanh\beta> 1$ and Erd\H{o}s--R\'enyi random graphs of average
degree $d$ when $d \tanh\beta> 1$.

\subsection{Standard background}
In the following subsection we recall some standard background on the
Ising model, Gibbs sampling and
Erd\H{o}s--R\'enyi random graphs.

\subsubsection{The Ising model}

The Ising model is perhaps the oldest and simplest discrete spin system
defined on graphs.
This model defines a distribution on labelings of the vertices of
the graph by $+$ and $-$.\vadjust{\goodbreak}

\begin{definition}
The (homogeneous) Ising model on a graph $G$ with inverse temperature
$\beta$ is a distribution on configurations $\{\pm\}^V$ such that
%
\begin{equation}
\label{eqdefising} P(\sigma)=\frac1{Z(\beta)}\exp\biggl(\beta\sum
_{\{v,u\}\in
E}\sigma(v)\sigma(u)\biggr),
\end{equation}
where $Z(\beta)$ is a normalizing constant.

More generally, we will be interested in the more general Ising models
defined by
%
\begin{equation}
\label{eqdefisinggeneral} P(\sigma)=\frac1{Z(
\beta)}\exp\bigl(H(\sigma)\bigr),
\end{equation}
where the Hamiltonian $H(\sigma)$ is defined as
\[
H(\sigma)=\sum_{\{v,u\}\in
E} \beta_{u,v}
\sigma(v)\sigma(u) + \sum_v h_v
\sigma(v),
\]
and where $h_v$ are arbitrary and $\beta_{u,v} \geq0$ for all $u$ and $v$.
In the more general case, we will write $\beta= \max_{u,v} \beta_{u,v}$.
\end{definition}

\subsubsection{Gibbs sampling}

The Gibbs sampler (also Glauber dynamics or heat bath) is a Markov
chain on configurations where a
configuration $\sigma$ is updated by choosing a vertex $v$ uniformly
at random and assigning it a spin according to the Gibbs
distribution conditional on the spins on $G-\{v\}$.
%
\begin{definition}
Given a graph $G=(V,E)$ and an inverse temperature $\beta$, the
Gibbs sampler is the discrete time Markov chain on $\{\pm\}^V$ where given
the current configuration $\sigma$ the next configuration $\sigma'$
is obtained
by choosing a vertex $v$ in $V$ uniformly at random and:
\begin{itemize}
\item
Letting $\sigma'(w) = \sigma(w)$ for all $w \neq v$.
\item
$\sigma'(v)$ is
assigned the spin $+$ with probability
\[
\frac{\exp(h_v + \sum_{u\dvtx(v,u)\in
E}\beta_{u,v} \sigma(u))}{\exp(h_v + \sum_{u\dvtx(v,u)\in
E}\beta_{u,v} \sigma(u))+\exp(-h_v-\sum_{u\dvtx(v,u)\in E} \beta_{u,v}
\sigma(u))}.
\]
%
\end{itemize}
\end{definition}
%

We will be interested in the time it takes the dynamics to get close
to distributions (\ref{eqdefising})
and (\ref{eqdefisinggeneral}). The \textit{mixing time}
$\tau_{\mathrm{mix}}$ of the chain is defined as the number of steps needed
in order to guarantee that the chain, starting from an arbitrary
state, is within total variation distance $1/2e$ from the
stationary distribution. The mixing time has the property that for any
integer $k$ and initial configuration $x$,
%
\begin{equation}
\label{emixingSubmult}
\bigl\| P (X_{k \tau_{\mathrm{mix}}} =\cdot\mid X_0=x ) -
P (\cdot) \bigr\|_{\mathrm{TV}} \leq e^{-k}.
\end{equation}

It is well known that Gibbs sampling is a reversible Markov chain
with stationary distribution $P$. Let $1=\lambda_1 > \lambda_2 \geq
\cdots\geq\lambda_m \geq-1$ denote the eigenvalues of the
transition matrix of Gibbs sampling. The \textit{spectral gap} is
denoted by $\min\{1-\lambda_2,1-|\lambda_m|\}$ and the
\textit{relaxation time} $\tau$ is the inverse of the spectral gap. The
relaxation time can be given in terms of the Dirichlet form of the
Markov chain by the equation
%
\begin{equation}
\label{eqrelaxdefn} \tau=\sup \biggl\{
\frac{2\sum_\sigma P(\sigma)
(f(\sigma))^2}{\sum_{\sigma\neq\tau} Q(\sigma,\tau)
(f(\sigma)-f(\tau))^2}\dvtx \sum_\sigma P(\sigma) f(\sigma)
\neq0 \biggr\},
\end{equation}
where $f\dvtx\{\pm\}^V\rightarrow\mathbb{R}$ is any function on configurations,
$Q(\sigma,\tau)=P(\sigma)P(\sigma\rightarrow\tau)$ and $P(\sigma
\rightarrow\tau)$ is transition probability from $\sigma$ to
$\tau$. We use the result that for reversible Markov chains the
relaxation time satisfies
%
\begin{equation}
\label{eqtauandspectral} \tau\leq
\tau_{\mathrm{mix}}\leq\tau \biggl(1+\frac12 \log\Bigl(\min_\sigma P(
\sigma)^{-1}\Bigr) \biggr),
\end{equation}
where $\tau_{\mathrm{mix}}$ is the mixing time (see, e.g., \cite{AldousFillu})
and so, by bounding the relaxation time, we can bound the mixing time
up to a polynomial factor.

While our results are given for the discrete time Gibbs Sampler
described above, it will, at times, be convenient to consider the
continuous time version of the model. Here sites are updated at rate 1
by independent Poisson clocks. The two chains are closely related: the
relaxation time of the continuous time Markov chain is $n$ times the
relaxation time of the discrete chain; see, for example, \cite{AldousFillu}.




\subsubsection{Erd\H{o}s--R\'enyi random graphs and other models of graphs}

The Erd\H{o}s--R\'enyi random graph $G(n,p)$, is the graph with $n$ vertices
$V$ and random edges $E$ where each potential edge $(u,v) \in V \times V$
is chosen independently with probability~$p$. We take
$p=d/n$ where $d \geq1$ is fixed. In the case $d < 1$, it is well
known that
with high probability all components of $G(n,p)$ are of logarithmic
size which implies immediately that the dynamics mix in polynomial time
for all $\beta$. A random $d$-regular graph $\mathcal{G}(n,d)$ is a
graph uniformly chosen from all $d$-regular graphs on $n$ labeled vertices.

Asymptotically the local neighborhoods of $G(n,d/n)$ and $\mathcal
{G}(n,d)$ are trees.
In the later case it is a tree where every node has exactly $d-1$
offspring (except for the root which has $d$ off-springs). In the
former case it is essentially a Galton--Watson branching process with
offspring distribution which is essentially Poisson with mean $d-1$.
Recall that the tree associated with a Galton--Watson branching process
with offspring distribution $X$ is a random rooted tree defined as
follows: for every vertex in the tree its number of offspring vertices
is independent with distribution $X$.

\subsection{A general theorem}
Theorems~\ref{tmain} and~\ref{tER} are both proved as special cases
of the following theorem which may be of independent interest.
For a graph $G=(V,E)$ and vertex $v \in V$, we write\vadjust{\goodbreak} $B(v,R)$ for the
ball of radius $R$ around $v$, that is, the set of all vertices that
are of distance at most $R$ from $v$. We write $S(v,R) = B(v,R)
\setminus B(v,R-1)$ for the sphere of radius $R$ around $v$.

\begin{maintheorem}\label{tgeneral}
Let $G$ be a graph on $n\geq2$ vertices such that there exist
constants $R,T,\mathfrak{X}\geq1$ such that the following
three conditions holds
for all $v\in V$:
\begin{itemize}
\item\textup{Volume}: The volume of the ball $B(v,R)$ satisfies
$|B(v,R)|\leq\mathfrak{X}$.

\item\textup{Local mixing}: For any configuration $\eta$ on $S(v,R)$
the continuous time mixing time of the Gibbs sampler on $B(v,R-1)$ with
fixed boundary condition $\eta$ is bounded above by $T$.

\item\textup{Spatial mixing}: For each vertex $u\in S(v,R)$, define
%
\begin{equation}
\label{eqSM1} a_u=\sup_{\eta^+,\eta^-} P\bigl(
\sigma_v=+\mid\sigma_{S}=\eta^+\bigr)- P\bigl(
\sigma_v=+\mid\sigma_{S}=\eta^-\bigr),
\end{equation}
where the supremum is over configurations $\eta^+,\eta^-$ on $S(v,R)$
differing only at $u$ with $\eta_u^+= +,\eta_u^-= -$. Then
%
\begin{equation}
\label{eqSM} \sum_{u\in S(v,R)} a_u \leq
\frac14.
\end{equation}
\end{itemize}
Then starting from the all $+$ and all $-$ configurations in continuous
time the monotone coupling couples with probability at least $\frac78$
by time $T \lceil\log8\mathfrak{X} \rceil ( 3 + \log_2 n
)$.

It follows that the mixing time of the Gibbs sampler in
continuous time satisfies
\[
\tau_{\mathrm{mix}} \leq T \lceil\log8\mathfrak{X} \rceil ( 3 +
\log_2 n ),
\]
while the spectral gap satisfies
\[
\mbox{Gap} \geq \bigl( T \lceil\log8\mathfrak{X} \rceil \bigr)^{-1}
\log2.
\]
\end{maintheorem}
We will write $\Vol(R,\mathfrak{X})$ for the statement that $|B(v,R)|
\leq\mathfrak{X}$ for all $v \in V$,
write $\SM(R)$ for the statement that (\ref{eqSM}) holds for all $v
\in V$ and write $\LM(R,T)$ for the statement that
the continuous time mixing time of the Gibbs sampler on $B(v,R-1)$ is
bounded above by $T$ for any fixed boundary condition $\eta$.
Using this notation the theorem states that:
%
\begin{equation}
\label{eqgeneral} \Vol(R,\mathfrak{X}) \mbox{ and } \SM(R) \mbox{ and }
\LM(R,T) \implies\tau_{\mathrm{mix}} \leq T \lceil\log8\mathfrak{X} \rceil ( 3 +
\log_2 n ).\hspace*{-28pt}
\end{equation}

In the conclusion section of the paper we state a much more general
version of Theorem~\ref{tgeneral} which applies to general monotone
Gibbs distributions and allows us to replace the balls $B(v,R)$ with be
arbitrary sets containing $v$ [where $S(v,R)$ is replaced by the inner
vertex boundary of the set]. We note that the implication proven here
for monotone systems showing
\[
\mbox{Spatial mixing} \implies\mbox{Temporal mixing}\vadjust{\goodbreak}
\]
is stronger than that established in previous work \cite
{StrookZegarlinski92,MartinelliOliveri94a,Cesi01,DSVW04} where it
is shown that strong spatial mixing implies temporal mixing for graphs
with sub-exponential growth (strong spatial mixing says that the
quantity $a_u$ decays exponentially in the distance between $u$ and
$v$). In particular, Theorem~\ref{tgeneral} applies also to graphs
with exponential growth and for a very general choice of blocks. Both
Theorems~\ref{tmain} and~\ref{tER} deal with expanding graphs where
Theorem~\ref{tgeneral} is needed.

A different way to look at our result is as a strengthening of the
Dobrushin--Shlosman condition \cite{DobrushinShlosman85}. Stated in its
strongest form in \cite{Weitz05}, Theorem 2.5, it says that rapid
mixing occurs if the effect on the spin at a vertex $v$ of
disagreements on the boundary of blocks containing $v$ is
small---averaged over all blocks containing $v$---then the model has
uniqueness and the block dynamics mixes rapidly. Theorem~\ref{tveryGeneral} requires only that for each vertex there exists a
block such that the boundary effect is small. This is critical in
expanders and random graphs where the boundary of a block is
proportional to its volume.

Finally we note that applied to the $d$-dimensional lattice
Theorem~\ref{tgeneral} gives a new proof of exponential ergodicity
for the Glauber dynamics on the infinite lattice $\mathbb{Z}^d$
whenever $\beta<\beta_c$ as well as a mixing time of $O(\log n)$ on
the $d$-dimensional torus of side-length $n$. The spatial mixing
condition follows from a result Higuchi \cite{Higuchi93}. This was
previously shown in Theorem 3.1 of~\cite{MartinelliOliveri94a}.

\subsection{Proofs sketch}
We briefly discuss the main ideas in our proofs of Theorems
\ref{tgeneral},~\ref{tmain} and~\ref{tER}.

\subsubsection{\texorpdfstring{Theorem \protect\ref{tgeneral} and censoring}{Theorem 3 and censoring}}
The proof of Theorem~\ref{tgeneral} is based on considering the
monotone coupling of the continuous time dynamics starting with all $+$
and all $-$ states and showing that there exists a constant $s$ such
that at time $k s$, for all vertices $v$, the probability that the two
measures have not coupled at $v$ is at most $2^{-k}$.

In order to prove such a claim by induction, it is useful to censor the
dynamics from time $k s$ onward by not performing any updates outside a
ball of radius $R$ around $v$. Recent results of Peres and Winkler show
that doing so will result in a larger disagreement probability at $v$
than without any censoring.

For the censored dynamics we use the triangle inequality and compare
the marginal probability at $v$ for the two measures by comparing each
distribution to the stationary distribution at $v$ given the boundary
condition and then comparing the two stationary distributions at $v$
given the two boundary conditions.

By using $\LM(R,T)$ and running the censored dynamics for $T \lceil
\log8\mathfrak{X} \rceil$ time, we can ensure that the error of the
first type contributes at most $2/(8\mathfrak{X})$ in case where the
two boundary conditions are different and therefore at most
$2/(8\mathfrak{X})$ times the expected number of disagreements at the
boundary which is bounded by $2^{-k-2}$ by induction. By using $\SM
(R)$ and the induction hypothesis, we obtain that the expected
discrepancy between the distributions at $\sigma_v$, given the two
different boundary conditions, is at most $2^{-k-2}$. Combining the two
estimates yields the desired result. As this gives an exponential rate
of decay in the expected discrepancy it establishes a constant lower
bound on the spectral gap.

The proofs of Theorems~\ref{tmain} and~\ref{tER} follow from
(\ref{eqgeneral})
by establishing bounds on $\Vol,\SM$ and $\LM$.

\subsubsection{Bounding the volume}
The easiest step in both Theorems~\ref{tmain} and~\ref{tER} is to
establish $\Vol(R,\mathfrak{X})$.
For graphs of degree at most $d$, the volume grows as $O((d-1)^{R})$
and using arguments from \cite{MosselSly09} one can show that if $R =
(\log\log n)^n$, then for $G(n,d/n)$, one can take $\mathfrak{X}$ of
order $d^R \log n$.

\subsubsection{Spatial mixing bounds}
Establishing spatial mixing bounds relies on the fact that for trees
without external fields, this is a standard calculation. The presence
of external fields can be dealt with by using a lemma from \cite
{BeKeMoPe05}, which shows that the for Ising model on trees, the
difference in magnetization is maximized when there are no external fields.
A crucial tool which allows us to obtain results for nontree graphs is
the Weitz tree \cite{Weitz06}.
This tree allows us to write magnetization ratios for the Ising model
on general graphs using a related model on the tree. In \cite
{MosselSly09} it was shown that the Weitz tree can be used to
construct an efficient algorithm, different than Gibbs sampling, for
sampling Ising configurations under the conditions of Theorems \ref
{tmain} and~\ref{tER} [the running time of the algorithm is
$n^{1+C(\beta)}$ compared to $C(\beta) n \log n$ established here].

\subsubsection{Local mixing bounds}
In order to derive local mixing bounds, we generalize results
from \cite{BeKeMoPe05} on the mixing times in terms of
cut-width to deal with arbitrary external fields. Further, for the case
of Erd\H{o}s--R\'enyi random graphs and
$R= (\log\log n)^2$, we show that with high probability the cut width
is of order $\log n / \log\log n$.

\section{Proofs}
In this section we prove Theorems~\ref{tgeneral},~\ref{tmain}
and~\ref{tER} while the verification of the $\Vol,\SM$ and $\LM$
conditions is deferred to the following sections. We begin by recalling
the notion of monotone coupling and the result by Peres--Winkler on
censoring. We then proceed with the proof of the theorems.

\subsection{Monotone coupling}\label{smonotoneCoupling}
For two configurations $X,Y\in\{ -, + \}^V$, we let $X \geq Y$
denote that $X$ is greater than or equal to $Y$ pointwise. When all the
interactions $\beta_{ij}$ are positive, it is well known that the
Ising model\vadjust{\goodbreak} is a monotone system under this partial ordering; that is,
if $X \geq Y$ then
\[
P (\sigma_v = + \mid \sigma_{V \setminus\{ v \}} = X_{V \setminus
\{ v \} } )
\geq P (\sigma_v = + \mid \sigma_{V \setminus\{
v \}} =Y_{V \setminus\{ v \} } ).
\]

As it is a monotone system, there exists a coupling of Markov chains $\{
X^x_t\}_{x\in\{ -,+\}^V } $ such that marginally, each has the law of
the Gibbs sampler with starting configurations $X_0^x = x$, and
further, that if $x \geq y$, then for all $t$, $X_t^x \geq X_t^y$. This
is referred to as the monotone coupling and can be constructed as
follows: let $v_1,\ldots$ be a random sequence of vertices updated by
the Gibbs sampler and associated with them i.i.d. random variables
$U_1,\ldots\,$, distributed as $U[0,1]$, which determine how the site is
updated. At the $i$th update, the site $v_i$ is updated to $+$ if
\[
U_i \leq\frac{\exp(h_v + \sum_{u\dvtx(v,u)\in
E}\beta_{u,v} \sigma(u))}{\exp(h_v + \sum_{u\dvtx(v,u)\in
E}\beta_{u,v} \sigma(u))+\exp(-h_v-\sum_{u\dvtx(v,u)\in E} \beta_{u,v}
\sigma(u))}
\]
and to $-$ otherwise. It is well known that such transitions preserve
the partial ordering which guarantees that if $x \geq y$, then $X_t^x
\geq X^y_t$ by the monotonicity of the system. In particular, this
implies that it is enough to bound the time taken to couple from the
all $+$ and all $-$ starting configurations.

\subsection{Censoring}

In general it is believed that doing more updates should lead to a more
mixed state. For the ferromagnetic Ising model and other monotone
systems, this intuition was proved by Peres and Winkler. They showed
that starting from the all $+$ (or all $-$) configurations, adding
updates only improves mixing. More formally they proved the following
proposition.

\begin{proposition}
Let $u_1,\ldots, u_m$ be a sequence of vertices, and let $i_1,\ldots,
i_l$ be a strictly increasing subsequence of $1,\ldots,m$. Let $X^+$
(resp., $X^{-}$) be a random configuration constructed by starting from
the all $+$ (resp., all $-$) configuration and running Gibbs updates
sequentially on $u_1,\ldots,u_m$. Similarly let $Y^+$ (resp., $Y^{-}$)
be a random configuration constructed by starting from the all $+$
(resp., all $-$) configuration and running Gibbs updates sequentially on
the vertices $u_{i_1},\ldots,u_{i_m}$. Then
\[
Y^{-} \preccurlyeq X^{-} \preccurlyeq X^{+}
\preccurlyeq Y^{+},
\]
where $A \preccurlyeq B$ denotes that $A$ stochastically dominates $B$
in the partial ordering of configurations.
\end{proposition}

This result in fact holds for random sequences of vertices of random
length and random subsequences, provided the choice of sequence is
independent of the choices that the Gibbs sampler makes. The result
remains unpublished, but its proof can be found in \cite{Peres05}.

\subsection{\texorpdfstring{Proof of Theorem \protect\ref{tgeneral}}{Proof of Theorem 3}}


Let $X_t^{+},X_t^{-}$, denote the Gibbs sampler on $G$ started,
respectively, from the all $+$ and $-$ configurations,
coupled using the monotone coupling described in Section\vadjust{\goodbreak} \ref
{smonotoneCoupling}. Fix some vertex $v\in
G$. We will define two new censored chains $Z_t^+$ and $Z_t^-$ starting
from the all $+$ and all $-$ configurations, respectively. Take $S \geq
0$ to be some arbitrary constant. Until time $S$ we set both $Z_t^+$
and $Z_t^-$ to be simply equal to $X_t^{+}$ and $X_t^{-}$, respectively.
After time $S$ all updates outside of $B(v,R-1)$ are censored; that is,
$Z_t^+$ and $Z_t^-$ remain unchanged on $V \setminus B(v,R-1)$ after
time $S$, but inside $B(v,R-1)$ share all the same updates with
$X_t^{+}$ and $X_t^{-}$.

In particular, this means that for $Z^+_t$ and $Z^-_t$ the spins on
$S(v,R)$ are fixed after time $S$.
By monotonicity of the updates we have $Z_t^+ \geq Z_t^-$ and $X_t^+
\geq
X_t^- $ for
all $t$. After time $S$ the censored processes are simply the Gibbs
sampler on $B(v,R-1)$ with boundary condition $X_S^\pm(S(v,R))$. By
assumption we have that the mixing time of this dynamics is bounded
above by $T$ and by equation (\ref{emixingSubmult}). If $t = T \lceil
\log8\mathfrak{X}\rceil$, then
%
\begin{equation}
\label{emixingApprox}\qquad \bigl\llvert P \bigl(Z^+_{S+t}(v)=+\mid
\mathcal{F}_S \bigr) - P \bigl(\sigma_v=+\mid
\sigma_{S(v,R)} = X^+_S\bigl(S(v,R)\bigr) \bigr) \bigr\rrvert
\leq\frac1{8\mathfrak{X}},
\end{equation}
and similarly for $Z^-$ where $\mathcal{F}_S$ denotes the
sigma-algebra generated by the updates up to time $S$. Now
%
\begin{eqnarray}
\label{egeneralA1}
&&
P \bigl( Z^+_{S+t}(v) \neq Z^-_{S+t}(v)
\mid\mathcal{F}_S \bigr)
\nonumber
\\
&&\qquad =P \bigl( Z^+_{S+t}(v) =+ \mid\mathcal{F}_S \bigr) - P
\bigl( Z^-_{S+t}(v) =+ \mid\mathcal{F}_S \bigr)
\nonumber\\[-8pt]\\[-8pt]
&&\qquad = I \bigl(X^+_S\bigl(B(v,R)\bigr)\neq X^-_S
\bigl(B(v,R)\bigr) \bigr)
\nonumber
\\
&&\qquad\quad{} \times \bigl[P \bigl( Z^+_{S+t}(v) =+ \mid\mathcal{F}_S
\bigr) - P \bigl( Z^-_{S+t}(v) =+ \mid\mathcal{F}_S
\bigr) \bigr],\nonumber
\end{eqnarray}
since if $X^+_S(B(v,R)) = X^-_S(B(v,R))$, then the censored processes
remains equal within $B(v,R)$ for all time as they receive the same
updates. Now we split up the right-hand side by the triangle
inequality:
%
\begin{eqnarray}
\label{egeneralA}
&&
I \bigl(X^+_S\bigl(B(v,R)\bigr)\neq
X^-_S\bigl(B(v,R)\bigr) \bigr)
\nonumber
\\
&&\quad{} \times \bigl[P \bigl( Z^+_{S+t}(v) =+ \mid\mathcal{F}_S
\bigr) - P \bigl( Z^-_{S+t}(v) =+ \mid\mathcal{F}_S
\bigr) \bigr]
\nonumber
\\
&&\qquad\leq I \bigl(X^+_S\bigl(B(v,R)\bigr)\neq X^-_S
\bigl(B(v,R)\bigr) \bigr)
\nonumber
\\
&&\qquad\quad{} \times \bigl[ \bigl\llvert P \bigl(Z^+_{S+t}(v)=+\mid
\mathcal{F}_S \bigr) - P \bigl(\sigma_v=+\mid
\sigma_{S(v,R)} = X^+_S\bigl(S(v,R)\bigr) \bigr) \bigr\rrvert
\\
&&\hspace*{15.5pt}\qquad\quad{} + \bigl| P \bigl(\sigma_v=+\mid\sigma_{S(v,R)} =
X^+_S\bigl(S(v,R)\bigr) \bigr)
\nonumber
\\
&&\hspace*{30pt}\qquad\quad{} - P \bigl(\sigma_v=+\mid\sigma_{S(v,R)} =
X^-_S\bigl(S(v,R)\bigr) \bigr) \bigr|
\nonumber
\\
&&\hspace*{16.3pt}\qquad\quad{} + \bigl\llvert P \bigl(Z^-_{S+t}(v)=+\mid\mathcal{F}_S
\bigr) - P \bigl(\sigma_v=+\mid\sigma_{S(v,R)} =
X^-_S\bigl(S(v,R)\bigr) \bigr) \bigr\rrvert \bigr].\nonumber
\end{eqnarray}
Now
%
\begin{eqnarray}
\label{egeneralB}
&&
E I \bigl(X^+_S\bigl(B(v,R)\bigr)\neq
X^-_S\bigl(B(v,R)\bigr) \bigr)
\nonumber
\\
&&\quad{}\times \bigl\llvert P \bigl(Z^+_{S+t}(v)=+\mid\mathcal{F}_S
\bigr) - P \bigl(\sigma_v=+\mid\sigma_{S(v,R)} =
X^+_S\bigl(S(v,R)\bigr) \bigr)\bigr\rrvert
\nonumber
\\
&&\qquad \leq\frac1{8\mathfrak{X}} E I \bigl(X^+_S\bigl(B(v,R)\bigr)\neq
X^-_S\bigl(B(v,R)\bigr) \bigr)
\\
&&\qquad \leq\frac1{8\mathfrak{X}} \sum_{u \in B(v,R)} P
\bigl(X^+_S(u)\neq X^-_S(u) \bigr)
\nonumber
\\
&&\qquad \leq\frac18 \max_{u\in V} P \bigl(X^+_S(u)\neq
X^-_S(u) \bigr),\nonumber
\end{eqnarray}
where the first inequality follows from equation (\ref
{emixingApprox}), the second by a union bound and the final inequality
follows from the volume assumption, and similarly for $Z^-$.

If $\eta^+ \geq\eta^-$ are two configurations on $S(v,R)$ which
differ only on the set $U\subseteq S(v,R)$, then by changing the
vertices one at a time by the spatial mixing condition, we have that
\[
P\bigl(\sigma_v=+\mid\sigma_{\Lambda}=\eta^+\bigr)- P\bigl(
\sigma_v=+\mid\sigma_{\Lambda}=\eta^-\bigr) \leq\sum
_{u\in U} a_u.
\]
It follows that
%
\begin{eqnarray}
\label{egeneralC}
&&E \bigl| P \bigl(\sigma_v=+\mid\sigma_{S(v,R)}
= X^+_S\bigl(S(v,R)\bigr) \bigr)
\nonumber
\\
&&\quad\hspace*{0pt}{} - P \bigl(\sigma_v=+\mid\sigma_{S(v,R)} =
X^-_S\bigl(S(v,R)\bigr) \bigr) \bigr|
\\
&&\qquad \leq\E\sum_{u \in B(v,R)} a_u I
\bigl(X^+_S(u)\neq X^-_S(u) \bigr) \leq\frac14
\max_{u\in V} P \bigl(X^+_S(u)\neq X^-_S(u)
\bigr).\nonumber
\end{eqnarray}
Combining equations (\ref{egeneralA1}), (\ref{egeneralA}), (\ref
{egeneralB}) and (\ref{egeneralC}), we have that
\[
P \bigl( Z^+_{S+t}(v) \neq Z^-_{S+t}(v) \bigr) \leq\frac12
\max_{u\in V} P \bigl(X^+_S(u)\neq X^-_S(u)
\bigr).
\]
By the censoring lemma, we have that $Z_t^+ \succcurlyeq X_t^+
\succcurlyeq
X_t^- \succcurlyeq Z_t^-$, and so
\[
P \bigl( X^+_{S+t}(v) \neq X^-_{S+t}(v) \bigr)\leq P \bigl(
Z^+_{S+t}(v) \neq Z^-_{S+t}(v) \bigr).
\]
Combining the previous two equations and taking a maximum over $v$, we
have that
%
\begin{equation}
\label{econtraction} \max_{u\in V} P \bigl( X^+_{S+t}(u) \neq
X^-_{S+t}(u) \bigr) \leq \frac12 \max_{u\in V} P
\bigl(X^+_S(u)\neq X^-_S(u) \bigr).
\end{equation}
Now $S$ is arbitrary, so we iterate equation (\ref{econtraction}) to
get that
\[
\max_{u\in V} P \bigl( X^+_{t(3+\lceil\log_2 n \rceil)}(u) \neq X^-_{t(3+\lceil\log_2 n \rceil)}(u)
\bigr) \leq2^{-3-\lceil\log
_2 n \rceil} \leq\frac1{2en}.
\]
Taking a union bound over all $u\in V$, we have that
\[
P \bigl( X^+_{t(3+\lceil\log_2 n \rceil)} \not\equiv X^-_{t(3+\lceil\log_2 n \rceil)} \bigr) \leq
\frac1{2e},
\]
and so the mixing time is bounded above by $T \lceil\log8\mathfrak
{X} \rceil ( 3 + \log_2 n  )$. Since the expected number
of disagreements, and hence the total variation distance from
stationarity, decays exponentially with a rate of at least $t^{-1}\log
2$, that is,
\[
E \# \bigl\{ u \in V\dvtx X^+_{s}(u) \neq X^-_{s}(u) \bigr\}
\leq2 n e^{-s t^{-1} \log2},
\]
it follows\vspace*{1pt} by standard results (see, e.g., Corollary 12.6 of \cite
{LevPerWil09}) that the spectral gap of the chain is bounded below by $t^{-1}
\log2$.

\subsection{\texorpdfstring{Proofs of Theorems \protect\ref{tmain} and \protect\ref{tER}}{Proofs of Theorems 1 and 2}}
We now prove Theorems~\ref{tmain} and~\ref{tER}, except
for the result for infinite graphs which will be proven in Section \ref
{sinfinite}.
Theorem~\ref{tmain} follows from (\ref{eqgeneral}) and the
following lemmas.
%
\begin{lemma} \label{lemmainVol}
Let $G=(V,E)$ be a graph of maximal degree $d$. Then $\Vol(R,\mathfrak
{X})$ holds with
\[
\mathfrak{X}= 1+d\sum_{\ell=1}^R(d-1)^{\ell-1}.
\]
\end{lemma}
%
\begin{lemma} \label{lemmainLM}
Let $G=(V,E)$ be a graph of maximal degree $d$, and consider the
ferromagnetic Ising model on $G$ with arbitrary external fields.
Then $\LM(R,T)$ holds with
\[
T=80 d^3 \mathfrak{X}^3 e^{5\beta d(\mathfrak{X}+1)},\qquad
\mathfrak {X}= 1+d\sum_{\ell=1}^R(d-1)^{\ell-1}.
\]
\end{lemma}
%
\begin{lemma} \label{lemmainSM}
Let $G=(V,E)$ be a graph with maximum degree $d$, and let $v \in V$.
Suppose that $(d-1)\tanh\beta< 1$. Let $R$ be an integer large enough
so that
%
\begin{equation}
\label{ediameterBound} \frac{d(d-1)^{R-1} \tanh^{R} \beta}{1-(d-1)\tanh\beta} \leq\frac1{4}.
\end{equation}
Then $\SM(R)$ holds.
\end{lemma}

We note that Lemma~\ref{lemmainVol} is trivial. As for Lemma \ref
{lemmainLM}, it is easy to prove a bound with a finite $T$ depending
on $R$, only assuming all external fields are bounded. We provide an
analysis with a tighter bound which applies also when the external
fields are not bounded. The proof is based on cut-width. The main step
is proving Lemma~\ref{lemmainSM}, which uses recursions on trees, a
comparison argument from \cite{BeKeMoPe05} and the Weitz tree.

The upper bound in Theorem~\ref{tER} follows from (\ref{eqgeneral})
and the following lemmas.

\begin{lemma} \label{lemERVol}
Let $G$ be a random graph distributed as $G(n,d/n)$.
Then $\Vol(R,\mathfrak{X})$ holds with high probability over $G$ with
\[
R = (\log\log n)^2,\qquad \mathfrak{X} = d^{R} \log n.\vadjust{\goodbreak}
\]
\end{lemma}
%
\begin{lemma} \label{lemERLM}
Let $G$ be a random graph distributed as $G(n,d/n)$ where $d$ is fixed.
There exists a constant $C(d)$
such that for $\LM(R,T)$ holds with high probability over $G$ with
\[
R = (\log\log n)^2,\qquad T = e^{10 \beta C(d) {\log n}/{\log
\log n}}.
\]
\end{lemma}
%
\begin{lemma} \label{lemERSM}
Let $G$ be a random graph distributed as $G(n,d/n)$ where $d$ is fixed
and $d \tanh\beta< 1$.
Then $\SM(R,T)$ holds with high probability over $G$ with $R = (\log
\log n)^2$.
\end{lemma}
The main challenge in extending the proof from bounded degree graphs to
$G(n,d/n)$
is obtaining a good enough control on the local geometry of the graph.
In particular, we obtain very tight tail estimates on the cut-width of
a Galton--Watson tree with Poisson offspring\vspace*{1pt} distribution of $(\log
\log n)^2$ levels.
A lower bound on the mixing time of $n^{1+\Omega(1/{\log\log
n})}$ was shown in~\cite{MosselSly09} by analyzing large star
subgraphs on $G(n,d/n)$. Recall that a star is a graph which is a
rooted tree with depth $1$ and that an Erd\H{o}s--R\'enyi random graph
with high probability there are stars with degree $\Omega(\frac{\log
n}{\log\log n})$.

\section{Volume growth}
We begin with verification of the Volume growth condition. Since
Lemma~\ref{lemmainVol} is trivial, this section will be devoted to
the proof of Lem\-ma~\ref{lemERVol} and other geometric properties of
random graphs. The reader who is interested in the proof of
Theorem~\ref{tmain} only may skip the remainder of this section.

The results stated in the section will require the notion of
\textit{tree excess}. For a graph $G$ we let $t(G)$ denote the
\textit{tree excess} of $G$, that is,
\[
t(G) = |E| - |V| + 1.
\]
Note that the second item of the following lemma implies the statement
of Lem\-ma~\ref{lemERVol}.

\begin{lemma}\label{lERGraphProperties}
Let $d$ be fixed, and let $G$ be a random graph distributed as
$G(n,d/n)$. The following hold with high probability over $G$ when $R=(
\log\log n)^2$ for all $v\in G$:
\begin{itemize}
\item
$B(v,R)$ has a spanning tree $T(v,R)$ which is stochastically dominated
by a Galton--Watson branching process
with offspring distribution $\operatorname{Poisson}(d)$.
\item The tree excess satisfies $t(v,R)\leq1$.
\item The volume of $B(v,R)$ is bounded by
\[
\bigl|B(v,R)\bigr|\leq d^{R} \log n.
\]
\end{itemize}
\end{lemma}

\begin{pf}
We construct a spanning tree $T(v,R)$ of $B(v,R)$ in a standard
manner. Take some arbitrary ordering of the vertices of $G$. Start
with the vertex $v$ and attach it to all its neighbors in $G$. Now
take the minimal vertex in $S(v,1)$, according to the ordering, and
attach it to all its neighbors in G which are not already in the
tree. Repeat this for each of the vertices in $S(v,1)$ in
increasing order. Repeat this for $S(v,2)$ and continue until
$S(v,R-1)$ which completes $T(v,R)$. By construction this is a
spanning tree for $B(v,R)$. The construction can be viewed as a
breadth first search of $B(v,R)$ starting from $v$ and exploring
according to the vertex ordering. By a standard argument $T(v,R)$ is
stochastically dominated
by a Galton--Watson branching process with offspring distribution
$\operatorname{Poisson}(d)$ with $R$ levels thus proving the first statement.

Since the volume of $B(v,R)$ equals the volume of $T(v,R)$, it suffices
to bound the later.
For this we use a variant of an argument from \cite{MosselSly09}.
We let $Z(r)$ denote the distribution of the
volume of a Galton--Watson tree of depth $r$ with off spring
distribution $N$, where $N$ is $\operatorname{Poisson}(d)$.
We claim that for all $t > 0$, it holds that
%
\begin{equation}
\label{eqexpmoment} \sup_r E\bigl[\exp\bigl(t
Z_r d^{-r}\bigr)\bigr] < \infty.
\end{equation}
Writing $s = s(t)$ for the value of the supremum, if follows from
Markov's inequality that
\[
s \geq P\bigl[Z_R \geq R^d \log n\bigr] \exp(t \log n)
\]
and so
\[
P\bigl[Z_R \geq R^d \log n\bigr] \leq s \exp(-t \log n),
\]
which is smaller than $o(1/n)$ if $t > 1$. This implies that $B(v,R)
\leq R^d \log n$ for all $v$ by a union bound and proves the second
statement of the lemma.

For (\ref{eqexpmoment}), let $N_i$ be independent copies of $N$
and note that
%
\begin{eqnarray}
\label{eqexpmomentrecursion} E\exp(t
Z_{r+1}) &=& E\exp\Biggl(\sum_{i=0}^{Z_r}
t d^{-(r+1)} N_i\Biggr) \nonumber\\
&=& E\Biggl[E\Biggl[\exp\Biggl(\sum
_{i=0}^{Z_r} t d^{-(r+1)} N_i\Biggr) \Bigm|
Z_r\Biggr]\Biggr]
\nonumber\\[-8pt]\\[-8pt]
&=& E\bigl[\bigl(E\bigl[\exp\bigl(t d^{-r+1} N\bigr)\bigr]
\bigr)^{Z_r}\bigr]
\nonumber
\\
&=& E \exp\bigl(Z_r \log\bigl(E\exp\bigl(t d^{-(r+1)} N\bigr)
\bigr)\bigr),\nonumber
\end{eqnarray}
which recursively relates the exponential moments of $Z_{r+1}$ to
the exponential moments of $Z_r$. In particular since all the
exponential moments of $Z_1$ exist, $E \exp(t Z_{r})<\infty$ for all
$t$ and $r$. When $0 < s \leq1$
%
\begin{equation}
\label{eqexpmomentbound} E\exp(s N) =
\sum_{i=0}^\infty\frac{s^i E N^i}{i!} \leq1 +
sd + s^2 \sum_{i=2}^\infty
\frac{E N^i}{i!} \leq\exp\bigl(sd(1+\alpha s)\bigr)
\end{equation}
provided $\alpha$ is sufficiently large. Now fix a $t$ and let
$t_r=t\exp(2\alpha t \sum_{i=r+1}^\infty d^{-i})$. For some
sufficiently large $j$ we have that $\exp(2\alpha t
\sum_{i=r+1}^\infty d^{-i})<2$ and $t_r d^{-(r+1)}<1$ for all $r\geq
j$. Then for $r\geq j$ by equations (\ref{eqexpmomentrecursion})
and (\ref{eqexpmomentbound}),
\begin{eqnarray*}
E\exp\bigl(t_{r+1} Z_{r+1} d^{-(r+1)}\bigr) &=& E \exp
\bigl(\log\bigl(E\exp\bigl(t_{r+1} d^{-(r+1)} N_i
\bigr)\bigr)Z_r\bigr)
\\
&\leq& E\exp\bigl(t_{r+1}\bigl(1+\alpha t_{r+1}
d^{-(r+1)}\bigr) Z_{r} d^{-r}\bigr)
\\
&\leq& E\exp\bigl(t_{r+1}\bigl(1+2\alpha t d^{-(r+1)}\bigr)
Z_{r} d^{-r}\bigr)
\\
&\leq& E\exp\bigl(t_{r} Z_{r} d^{-r}\bigr)
\end{eqnarray*}
and so
\[
\sup_{r\geq j} E\exp\bigl(t Z_{r} d^{-r}\bigr) \leq
\sup_{n\geq j} E\exp\bigl(t_{r} Z_{r}
d^{-r}\bigr) = E\exp\bigl(t_{j} Z_{j}
d^{-j}\bigr) < \infty,
\]
which completes the proof of (\ref{eqexpmoment}).

It remains to bound the tree excess. In the construction of $T(v,R)$
there may be some edges in
$B(v,R)$ which are not explored and so are not in $T(v,R)$.
Each edge between $u,w\in V(v,R)$ which is not explored
in the construction of $T(v,R)$ is present in $B(v,R)$ independently
with probability $d/n$. There are at most $d^{2R}$ unexplored edges and
\[
P\bigl(\operatorname{Binomial}\bigl(d^{2R},d/n\bigr) > 1\bigr) \leq
d^{4R} (d/n)^2 \leq n^{-2+o(1)}
\]
for any fixed $d$. So by a union bound with high probability we have
that $t(v,R) \leq1$ for all $v$.
\end{pf}

\section{Local mixing}
In this section we prove Lemmas~\ref{lemmainLM} and~\ref{lemERLM}.
The proof that the local mixing condition holds for graphs of bounded
degree, bounded volume and bounded external field is standard. Indeed
the reader who is interested in Theorem~\ref{tmain} for models with
bounded external fields may skip this section.

\subsection{Cut-width bounds}
The main tool in bounding the mixing time will be the notion of
cut-width used in \cite{BeKeMoPe05}.
Recall that the \textit{cut-width} of a finite graph $G=(V,E)$
\[
\min_{\pi\in S(n)} \max_{1 \leq i \leq n-1} \bigl|\{ v_{\pi(j)}\dvtx j \leq i \}
\times\{ v_{\pi(j)}\dvtx j > i \} \cap E\bigr|,
\]
where the minimum is taken over all permutations of the labels of the
vertices $v_1,\ldots,v_n$ in $V$.
%
%

We will prove the following result which generalizes the results
of \cite{BeKeMoPe05} to the case with boundary conditions. The proof
follow the ones given in \cite{BeKeMoPe05} and \cite{Martinelli99}.

\begin{lemma}\label{lcutwidth}
Consider the Ising model on G with interaction strengths bounded by
$\beta$, arbitrary external field, cut-width $\mathcal{E}$ and
maximal degree $d$. Then the relaxation time of the discrete time Gibbs
sampler is at most $n^2e^{4\beta(\mathcal{E}+d)}$.
\end{lemma}

\begin{pf}
We follow the notation of \cite{KeMoPe01}. Fix an ordering ``$<$'' of
the vertices in $V$ which achieves the cut-width. Define a canonical
path $\gamma(\sigma,\eta)$ between two configurations $\sigma,\eta
$ as follows: let $v_1 < v_2 < \cdots< v_\ell$ be the vertices on
which $\sigma$ and $\eta$ differ. The $k$th configuration in the path
$\eta=\sigma^{(0)}, \sigma^{(1)},\ldots, \sigma^{(\ell)}$ is
defined by $\sigma^{(k)}_v = \sigma_v$ for $v \leq v_k$ and $\sigma^{(k)}_v = \eta_v$ for $v>v_k$. Then by the method of canonical paths
(see, e.g., \cite{JerrumSinclair89,Martinelli99}), the relaxation
time is bounded by
\[
\tau\leq n \sup_e \sum_{\sigma,\eta\dvtx e\in\gamma(\sigma,\eta)}
\frac{P(\sigma)P(\eta)}{Q(e)},
\]
where the supremum is over all pairs of configurations $e=(x,y)$ which
differ at a single vertex and where $e\in\gamma(\sigma,\eta)$
denotes that $x$ and $y$ are consecutive configurations in the
canonical path $\gamma(\sigma,\eta)$ and $Q((x,y))=P(x)P(x\to y)$.

Let $e=(x,y)$ be a pair of configurations which differ only at $v$. For
a pair of configurations $\sigma,\eta$ let $\varphi_e(\sigma,\eta
)$ denote the configuration which is given by $\varphi_e(\sigma,\eta
)_{v'}=\eta_{v'}$ for $v'<v$ and $\varphi_e(\sigma,\eta
)_{v'}=\sigma_{v'}$ for $v' \geq v$. Further, by construction we have
that for any $u\in V$, that the unordered pairs $\{\sigma_u,\eta_u\}$
and $\{x_u,\varphi_e(\sigma,\eta)_{u}\}$ are equal, and so
\[
\sum_u h_u (\sigma_u
+ \eta_u) =\sum_u
h_u \bigl(x_u + \varphi_e(\sigma ,
\eta)_{u}\bigr).
\]
Also if $(u,u')\in E$ is such that $u,u'<v$ or $u,u'>v$, then again by
the labeling we have that
\[
\sigma_u \sigma_{u'}+ \eta_u
\eta_{u'} = x_u x_{u'}+ \varphi_e(
\sigma,\eta)_{u}\varphi_e(\sigma,\eta)_{u'}.
\]
Combining these results we have that
\begin{eqnarray*}
&&\frac{P(\sigma)P(\eta)}{P(x)P(\varphi_e(\sigma,\eta))}
\\
&&\qquad= \frac{\exp [ \sum_{\{u,u'\}\in
E} \beta_{u,u'} (\sigma_u \sigma_{u'} + \eta_u \eta_{u'})+ \sum_u
h_u (\sigma_{u} + \eta_{u}) ]}{\exp [ \sum\beta_{u,u'}
(x_u x_{u'} + \varphi_e(\sigma,\eta)_u \varphi_e(\sigma,\eta
)_{u'})+ \sum h_u (x_{u} + \varphi_e(\sigma,\eta)_{u}) ]}
\\
&&\qquad\leq e^{4 \mathcal{E}(\beta)}
\end{eqnarray*}
as the only terms which don't cancel are those relating to edges
$(u,u')$ with $u<v<u'$ or $u'<v<u$, of which there are only $\mathcal{E}$.
A crude bound on the transition probabilities gives that
\[
P(x\to y) \geq\frac1{n} \frac{e^{h_v y_v - d \beta}}{e^{h_v y_v - d
\beta}+e^{-h_v y_v + d \beta}}.
\]
Then
\begin{eqnarray*}
\sum_{\sigma,\eta\dvtx e\in\gamma(\sigma,\eta)} \frac{P(\sigma
)P(\eta)}{Q(e)} &\leq&
e^{4 \mathcal{E}\beta} \frac1{P(x\to y)}\sum_{\sigma,\eta\dvtx
e\in\gamma(\sigma,\eta)} P\bigl(
\varphi_e(\sigma,\eta)\bigr)
\\
&\leq& n e^{4 \mathcal{E}\beta} \bigl(1+ e^{-2h_v y_v + 2d \beta}\bigr)\sum
_{\sigma,\eta\dvtx e\in\gamma(\sigma,\eta)} P\bigl(\varphi_e(\sigma,\eta)\bigr).
\end{eqnarray*}
The labeling is constructed such that for each $e$ and configuration
$z$ there is at most on pair $(\sigma,\eta)$ with $e\in\gamma
(\sigma,\eta)$ so that $\varphi_e(\sigma,\eta)=z$. Also we have
that $\varphi_e(\sigma,\eta)_{v}=\sigma_v=y_v$ and so
\begin{eqnarray*}
\sum_{\sigma,\eta\dvtx e\in\gamma(\sigma,\eta)} P\bigl(\varphi_e(\sigma ,
\eta)\bigr) &\leq&\sum_{\sigma\dvtx\sigma_v=y_v} P(\sigma) \leq
\frac{e^{h
y_v+d \beta}}{e^{h y_v+d \beta}+e^{-h y_v-d \beta}}\\
&=&\frac1{1+e^{-2h
y_v-2d \beta}},
\end{eqnarray*}
where the inequality holds and hence
\[
\tau\leq n^2 e^{4 \mathcal{E}\beta} \frac{1+ e^{-2h_v y_v + 2d
\beta}}{1+e^{-2h\sigma_v-2d \beta}} = n^2
e^{4 \mathcal{E}\beta
}\frac{1+ e^{-2h_v \sigma_v} e^{2d \beta}}{1+e^{-2h\sigma_v} e^{-2d
\beta}} \leq n^2 e^{4 \mathcal{E}\beta+4d\beta}
\]
as required.
\end{pf}

We now need to establish a bound to relate the relaxation time to the
mixing time. While we would like to apply equation (\ref
{eqtauandspectral}) directly to Lemma~\ref{lcutwidth}, if the
external fields go to infinity, the right-hand side of equation (\ref
{eqtauandspectral}) also goes to infinity. So that our results holds
for any external field, we establish the following lemma.

\begin{lemma}\label{lcutwidthMix}
Consider the Ising model on G with interaction strengths bounded by
$\beta$, cut-width $\mathcal{E}$, arbitrary external field and
maximal degree $d$. Then the mixing time of the Gibbs sampler satisfies
\[
\tau_{\mathrm{mix}} \leq80 n^3 e^{5\beta(\mathcal{E}+d)}.
\]
\end{lemma}


%

\begin{pf}
Define $\hb=3\log n + 6\beta\mathcal{E} + 4d\beta+10$, and let $U$
denote the set of vertices $U=\{v\in V\dvtx|h_v| \geq\hb\}$. These are
the set of vertices with external fields so strong that it
is\vspace*{1pt} highly unlikely that they are updated to a value other
than $\operatorname{sign}(h_v)$. Let $\tilde{G}$ denote the graph
induced by the vertex set $\tilde{V}=V\setminus U$, and let $\tilde{P}$
denote the Ising model with the same interaction strengths $\beta_{uv}$
but with modified external field
\[
\tilde{h}_v =h_v + \sum
_{u\in U\dvtx(u,v)\in E} \beta_{uv} \operatorname{sign}(h_u).
\]
This is, of course, just the original Ising model restricted to $\tilde
{V}$ with external field given by $\sigma_u=\operatorname{sign}(h_u)$ for
$u\in U$. We now analyze the continuous time Gibbs sampler of $\tilde
{P}$. By Lemma~\ref{lcutwidth} its relaxation time satisfies
\[
\tilde{\tau} \leq n e^{4\beta(\mathcal{E}+d)}
\]
since restricting to $\tilde{G}$ can only decrease the cut-width and
maximum degree and since the discrete and continuous relaxation times
differ by a factor of~$n$.
To invoke~(\ref{eqtauandspectral}), we bound $\min_\sigma\tilde
{P}(\sigma)$. By our construction, we have that
\[
\max_{v\in\tilde{V}} | \tilde{h}_v| \leq\hb+d\beta.
\]
Now
\[
\min_{\sigma\in\{+,-\}^{\tilde{V}}} \tilde{H}(\sigma) = \min_{\sigma} \sum
_{\{v,u\}\in
\tilde{E}} \beta_{u,v} \sigma(v)\sigma(u) + \sum
_{v\in\tilde{V}} h_v \sigma(v) \geq-n(2d\beta+
\hb)
\]
and similarly $ \max_{\sigma} \tilde{H}(\sigma) \leq n(2d\beta+\hb
)$. Now the normalizing constant $\tilde{Z}$ satisfies
\[
\tilde{Z}=\sum_{\sigma\in\{+,-\}^{\tilde{V}}} \exp\bigl(\tilde {H}(\sigma)
\bigr) \leq2^n \exp\bigl(n(2d\beta+\hb)\bigr),
\]
so finally
\[
\min_{\sigma\in\{+,-\}^{\tilde{V}}} \tilde{P}(\sigma)\geq\frac
{\min_{\sigma} \exp(\tilde{H}(\sigma))}{\tilde{Z}}\geq2^{-n}
\exp \bigl(-n(4d\beta+2\hb)\bigr).
\]
By equation (\ref{eqtauandspectral}) this implies that the mixing
time of the continuous time Gibbs sampler on $\tilde{P}$ satisfies
\begin{eqnarray*}
\tilde{\tau}_{\mathrm{mix}} &\leq&\tilde{\tau} \biggl(1+\frac12 \log\Bigl(
\min_\sigma \tilde{P}(\sigma)^{-1}\Bigr) \biggr) \\
&\leq& n
e^{4\beta(\mathcal
{E}+d)} \biggl( 1 + \frac12 n(\log2 + 2d\beta+ \hb) \biggr).
\end{eqnarray*}
We set $T=8n^2 \hb e^{4\beta(\mathcal{E}+d)} \geq4 \tilde{\tau}_{\mathrm{mix}}$.

We now return to the continuous time dynamics on all $G$. Let $\mathcal
{A}$ denote the event that every vertex in $u\in U$ is updated at least
once before time $T$. The probability that a vertex $u$ is updated by
time $T$ is $1-e^{-T}$, and so by a union bound,
\[
P(\mathcal{A}) \geq1 - ne^{-T}\geq1 -ne^{-\hb} \geq1 -
e^{-10}.
\]
Let $\mathcal{B}$ be the event that for every vertex $u\in U$, every
update up to time $2T$ updates the spin to $\operatorname{sign}(h_u)$. For a
single vertex $u\in U$ and any configuration $\sigma$ when $u$ is updated,
%
\begin{equation}
\label{eupdateBound} P \bigl(u \mbox{ is updated to } {-}\operatorname{sign}(h_u)
\bigr) \leq\frac
{e^{-|h_u| + d\beta}}{e^{-|h_u| + d\beta}+e^{|h_u| - d\beta}} \leq e^{-2\hb+ 2d\beta}.
\end{equation}
The number of updates in $U$ up to time $2T$ is distributed as a
Poisson random variable with mean $2T|U|$ so
\begin{eqnarray*}
P(\mathcal{B})
&\geq& P\bigl(\operatorname{Po}\bigl(2Tne^{-2\hb+ 2d\beta} \bigr) =
0\bigr)
\\
& = &e^{-2Tne^{-2\hb+ 2d\beta} }
\\
&\geq& 1 -2Tne^{-2\hb+ 2d\beta}
\\
&\geq& 1 - 8n^3 \hb e^{4\beta(\mathcal{E}+d)-2\hb+ 2d\beta}
\\
&=& 1 - 8 \hb e^{-\hb-10}
\\
&>& 1-8e^{-10},
\end{eqnarray*}
where the last inequality follows from the fact that $e^x>x$.

Let $X_t$ denote the Gibbs sampler with respect to $P$, and let $Y_t$
be its restriction to~$\tilde{V}$.
Conditioned on $\mathcal{A}$ and $\mathcal{B}$ by time $T$, every
vertex in $U$ has been updated, and it has been updated to $
\operatorname{sign}(h_u)$ and remains with this spin until time $2T$. For $T\leq
t\leq2T$ let $Y_t$ denote the Gibbs sampler on $\tilde{V}$ with
respect to $\tilde{P}$ with initial condition $Y_T = X_T(\tilde{V})$.
From time $T$ to $2T$, couple $X_t$ and $Y_t$ with the same updates
(i.e., inside $\tilde{V}$ the same choice of $\{v_i\}$ and $\{U_i\}$in
the notation of Section~\ref{smonotoneCoupling}). Then conditioned on
$\mathcal{A}$ and $\mathcal{B}$, we have that $Y_t = X_t(\tilde{V})$
for $T\leq t \leq2T$.

We can now use our bound on the mixing time of the Gibbs sampler with
respect to $\tilde{P}$. Since $T \geq4\tilde{\tau}_{\mathrm{mix}}$, by
equation (\ref{emixingSubmult}) we have that
%
\begin{equation}
\label{emixingRestricted}
\bigl\| P (Y_{2T}=\cdot ) - \tilde{P} (\cdot)
\bigr\|_{\mathrm
{TV}} \leq e^{-4}.
\end{equation}
Under the stationary measure $P$, it follows from equation (\ref
{eupdateBound}) that for any $u\in U$,
\[
P \bigl(\sigma_u = \operatorname{sign}(h_u) \bigr)
\geq1-e^{2|h_u| -
2d\beta}
\]
and hence by a union bound,
%
\begin{equation}
\label{estationaryCondition} P \bigl(\sigma_u =
\operatorname{sign}(h_u), \forall u \in U \bigr) \geq
1-ne^{2\hb- 2d\beta}
\end{equation}
and so
\[
\bigl\| P \bigl(\sigma\in\cdot\mid\sigma_u = \operatorname{sign}(h_u),
\forall u \in U \bigr) - P (\sigma\in\cdot) \bigr\|_{\mathrm{TV}} \leq
ne^{2\hb- 2d\beta}.
\]
Since the projection of $P$ onto $\tilde{V}$ conditioning on $\sigma_u = \operatorname{sign}(h_u)$ for all $u\in U$ is simply $\tilde{P}$, it
follows that
\begin{eqnarray*}
&&\bigl\| P (X_{2T}=\cdot ) - \tilde{P} (\cdot) \bigr\|_{\mathrm
{TV}}
\\
&&\qquad \leq P\bigl(\mathcal{A}^c\bigr) + P\bigl(\mathcal{B}^c
\bigr)\\
&&\qquad\quad{} + \bigl\| P \bigl(\sigma\in \cdot\mid\sigma_u =
\operatorname{sign}(h_u), \forall u \in U \bigr) - P(\sigma\in\cdot)
\bigr\|_{\mathrm{TV}}
\\
&&\qquad\quad{} + \bigl\| P (Y_{2T}\in\cdot ) - \tilde{P} (\sigma\in \cdot)
\bigr\|_{\mathrm{TV}}
\\
&&\qquad\leq9e^{-10}+ ne^{2\hb- 2d\beta} + e^{-4}
\\
&&\qquad\leq\frac1{2e},
\end{eqnarray*}
which establishes $2T$ as an upper bound on the mixing time $\tau_{\mathrm{mix}}$. By a crude bound, $\hb\leq10ne^{\beta(d+\mathcal{E})}$,
which establishes
\[
\tau_{\mathrm{mix}} \leq2T \leq8n^2 \hb e^{4\beta(\mathcal{E}+d)} \leq80
n^3 e^{5\beta(\mathcal{E}+d)}
\]
as required.
\end{pf}

\subsection{Proof of local mixing for graphs of bounded degree}
We can now prove Lemma~\ref{lemmainLM}.

\begin{pf*}{Proof of Lemma~\ref{lemmainLM}}
The proof follows immediately from Lemma~\ref{lcutwidthMix}, applied
to the balls $B(v,R)$, and noting that $\mathcal{E}$ is always smaller
than the number of vertices in the graph which is bounded by $\mathfrak{X}$.
\end{pf*}

\subsection{Cut-width in random graphs and Galton--Watson trees}
The main result we prove in this section is the following.

\begin{lemma}\label{lcutwidthGW}
For every $d$ there exists a constant $C'(d) $ such that the following hold.
Let $T$ be the tree given by the first $\ell$ levels of a
Galton--Watson branching process tree with $\operatorname{Poisson}(d)$ offspring
distribution. Then $\mathcal{E}(T)$, the cut-width of $T$ is
stochastically dominated by the distribution $C'\ell+ \operatorname{Po}(d)$.
\end{lemma}

Using this result, it is not hard to prove the upper bound on the local
mixing of Lemma~\ref{lemERLM}.

\begin{pf*}{Proof of Lemma~\ref{lemERLM}}
We first note that by Lemma~\ref{lERGraphProperties} with high
probability for all $v$, the tree excess of the ball $B(v,R)$ is at
most one. This implies that the cut-width of $B(v,R)$ is at most $1$
more than the cut-width of the spanning tree $T(v,R)$ of $B(v,R)$ whose
distribution is dominated by a Galton--Watson tree with Poisson
offspring distribution with mean $d$. We thus conclude by
Lemma~\ref{lcutwidthGW} that with high probability for all $v \in V$,
the distribution of the cut-width of $B(v,R)$ is bounded by $C'R +
\operatorname{Po}(d)$. Since the probability that $\operatorname{Po}(d)$ exceeds $c
\log n / \log\log n$ for large enough $c$ is of order $n^{-2}$, we
obtain by a union bound that with high probability for all $v$ it holds
that $B(v,R)$ has a cut-width of at most $(c+C') \log n / \log\log n$.
Similarly with high probability, the maximal degree in $G$ is of order
$\log n / \log\log n$. Recalling that $\mathfrak{X}$ is at most $d^R
\log n$ and applying Lemma~\ref{lcutwidthMix} yields the required
result.
\end{pf*}

The proof of Lemma~\ref{lcutwidthGW} follows by induction from the
following two lemmas.
%
\begin{lemma}\label{lcutwidthTreeClaim}
Let $T$ be a tree rooted at $\rho$ with degree $m$, and let
$T_1,\ldots,T_m$ be the subtrees connected to the root. Then the
cut-width of $T$ satisfies
\[
\mathcal{E}(T) \leq\max_i \mathcal{E}(T_i) + m +1 - i.
\]
\end{lemma}
\begin{pf}
For each\vspace*{1pt} subgraph $T_i$, let $u_1^{(i)},\ldots,u_{|V_i|}^{(i)}$ be a
sequence on vertices which achieves the cut-width $\mathcal{E}(T_i)$.
Concatenate these sequences as
\[
\rho,u_1^{(1)},\ldots,u_{|V_1|}^{(1)},u_1^{(2)},
\ldots, u_{|V_k|}^{(k)},
\]
which can easily be seen to achieve the bound $\max_i \mathcal
{E}(T_i) + k + 1 - i$.
\end{pf}

For a collection of random variables $Y_1,\ldots,Y_k$, the order
statistics is defined as the permutation of the values into increasing
order such that $Y_{(1)} \leq\cdots\leq Y_{(k)}$.
%
\begin{lemma}\label{lpoissonIdentity}
Let $X \sim\operatorname{Po}(d)$, and let $Y_1,\ldots,Y_{X}$ be an i.i.d.
sequence distributed as $\operatorname{Po}(d)$. There exists $C(d)$ such that
\[
W=X+\max_{1\leq i \leq X} Y_{(i)} - i
\]
is stochastically dominated by $C+ \operatorname{Po}(d)$.
\end{lemma}
\begin{pf}
The probability distribution of the Poisson is given by
$P(\operatorname{Po}(d) =w) = \frac{d^w e^{-d}}{w!}$
which decays faster than any exponential, so
\[
\frac{P(\operatorname{Po}(d) \geq w) }{P(\operatorname{Po}(d) =w) }\rightarrow1
\]
as $w\rightarrow\infty$. With this fast rate of decay, we can choose
$C=C(d)$ large enough so that the following hold:
\begin{itemize}
\item$C \geq6$ is even, and for $w\geq\Cdu $,
%
\begin{equation}
\label{eGWExposureAssum0} P\bigl(\operatorname{Po}(d)\geq w +1\bigr)\leq P\bigl(
\operatorname{Po}(d)= w\bigr);
\end{equation}
\item for all $w\geq0$,
%
\begin{equation}
\label{eGWExposureAssum1} \biggl(w + \Cdu \biggr) E2^X P\biggl(
\operatorname{Po}(d)\geq w + \Cdu \biggr)\leq\frac 1{100} P\bigl(
\operatorname{Po}(d)\geq w\bigr);
\end{equation}
\item for all $w\geq0$,
%
\begin{equation}
\label{eGWExposureAssum2}
P\biggl(\operatorname{Po}(d)\geq\biggl\lfloor\frac{w}{2}\biggr\rfloor +
C \biggr)^{3} \leq P\biggl(\operatorname{Po}(d)\geq w + \Cdu \biggr),
\end{equation}
which can be achieved since $\frac1{((\lfloor{w}/{2}\rfloor+
C)!)^3} \ll \frac1{(w + {C}/{2} )!}$;
%
\item for all $w\geq2$,
%
\begin{equation}
\label{eGWExposureAssum4} \biggl(w + \frac{C}{2} \biggr)^2 2^{2w+{3C}/{2}} P \biggl(
\operatorname{Po}(d)\geq\frac{C}{2} \biggr)^{\lfloor{w}/{2}\rfloor}\leq\frac1{100};\vadjust{\goodbreak}
\end{equation}
\item for $w\in\{0,1\}$,
%
\begin{equation}
\label{eGWExposureAssum5} P(W \geq w+C) \leq P \bigl(\operatorname{Po}(d)\geq w
\bigr).
\end{equation}
\end{itemize}
Observe that for $1\leq i \leq x$,
%
\begin{eqnarray}
\label{eGWExposureBinomial} P( Y_{(i)} \geq w\mid X=x)
&\leq& \pmatrix{x
\cr
x-i+1}P\bigl(\operatorname{Po}(d) \geq w\bigr)^{x-i+1}
\nonumber\\[-8pt]\\[-8pt]
&\leq& 2^x P\bigl(\operatorname{Po}(d)\geq w\bigr)^{x-i+1}\nonumber
\end{eqnarray}
since if $Y_{(i)} \geq w$ then there are at least $x-i+1$ of the $Y$'s
must be greater than or equal to $w$ and there are $ {x \choose x-i+1}$
such choices of the set. For any $y,z\geq0$, we have that
%
\begin{eqnarray}
\label{eGWExposureExpSum} P\bigl(\operatorname{Po}(d) = y\bigr)P\bigl(
\operatorname{Po}(d) = z\bigr) &=& \frac{d^y e^{-d}}{y!} \frac
{d^z e^{-d}}{z!}
\nonumber
\\
&=&\pmatrix{y+z
\cr
z} \frac{d^{y+z} e^{-2d}}{(y+z)!} \\
&\leq&2^{y+z}P\bigl(
\operatorname{Po}(d) = y+z\bigr)\nonumber
\end{eqnarray}
since ${y+z\choose z} \leq2^{y+z}$.

Fix a $w \geq2$. Then
%
\begin{eqnarray}
\label{eGWExposureA}
&&
P(W \geq w+C)
\nonumber
\\
&&\qquad= P\Bigl(X+\max_{1\leq i \leq X} Y_{(i)} - i \geq w+C\Bigr)
\nonumber
\\
&&\qquad\leq P\biggl(X>w+\frac{C}{2} \biggr) \nonumber\\[-8pt]\\[-8pt]
&&\qquad\quad{}+ \sum_{x=1}^{w+{C}/{2} } P
\Bigl(x+\max_{1\leq i \leq x} Y_{(i)} - i \geq w+C\mid X=x\Bigr)P(X=x)
\nonumber
\\
&&\qquad\leq\frac1{100}P(X=w) \nonumber\\
&&\qquad\quad{}+ \sum_{x=1}^{w+{C}/{2} } P
\Bigl(x+\max_{1\leq i \leq
x} Y_{(i)} - i \geq w+C\mid X=x\Bigr)P(X=x),\nonumber
\end{eqnarray}
where the final equality follows from equation (\ref
{eGWExposureAssum1}). Now
%
\begin{eqnarray}
\label{eGWExposureB}
&&\sum_{x=1}^{w+{C}/{2} } P
\Bigl(x+\max_{1\leq i \leq x} Y_{(i)} - i \geq w+C\mid X=x\Bigr)P(X=x)
\nonumber
\\
&&\qquad \leq\sum_{x=1}^{w+{C}/{2} } \sum
_{i=1}^x P(x+ Y_{(i)} - i \geq w+C\mid
X=x)P(X=x)
\nonumber
\\
&&\qquad = \sum_{x=1}^{w+{C}/{2} } \sum
_{j=1}^x P(Y_{(x-j+1)} \geq w -j + 1+C\mid
X=x)P(X=x)
\\
&&\qquad \leq\sum_{x=1}^{w+{C}/{2} } \sum
_{j=1}^x 2^x P\bigl(
\operatorname{Po}(d)\geq w-j+1+C\bigr)^{j} P(X=x)
\nonumber
\\
&&\qquad = \sum_{j=1}^{w+{C}/{2} } \sum
_{x=j}^{w+{C}/{2} } 2^x P\bigl(
\operatorname{Po}(d)\geq w-j+1+C\bigr)^{j} P(X=x),\nonumber
\end{eqnarray}
where line 3 follows by setting $j=x-i+1$, and line 4 follows from
equation~(\ref{eGWExposureBinomial}). We split this sum into 3 parts.
First we have that
%
\begin{eqnarray}
\label{eGWExposureC}
&&
\sum_{j=1}^{{C}/{2} }
\sum_{x=j}^{w+{C}/{2} } 2^x P\bigl(
\operatorname{Po}(d)\geq w-j+1+C\bigr)^{j} P(X=x)
\nonumber
\\
&&\qquad
\leq\frac{C}{2} \sum_{x=1}^{w+{C}/{2} } 2^x
P\biggl(\operatorname{Po}(d)\geq w + \frac{C}{2}\biggr) P(X=x)
\\
&&\qquad
\leq\frac{C}{2} E2^X P\biggl(\operatorname{Po}(d)\geq w + \frac{C}{2} \biggr)
\nonumber
\\
&&\qquad
\leq\frac1{100} P\bigl(\operatorname{Po}(d)\geq w\bigr),\nonumber
\end{eqnarray}
where the final equality follows from equation (\ref
{eGWExposureAssum1}). Second,
%
\begin{eqnarray}
\label{eGWExposureD}
&& \sum_{j={C}/{2} +1}^{\lfloor{w}/{2}\rfloor }
\sum_{x=j}^{w+{C}/{2} } 2^x P\bigl(
\operatorname{Po}(d)\geq w-j+1+C\bigr)^{j} P(X=x)
\nonumber
\\
&&\qquad \leq\biggl\lfloor\frac{w}{2}\biggr\rfloor
\sum_{x={C}/{2} +1}^{w+{C}/{2} } 2^x
P\biggl(\operatorname{Po}(d)\geq\biggl\lfloor\frac{w}{2}\biggr\rfloor + C\biggr)^{{C}/{2} } P(X=x)
\nonumber
\\
&&\qquad \leq\biggl\lfloor\frac{w}{2}\biggr\rfloor E2^X P\biggl(\operatorname{Po}(d)\geq\biggl\lfloor\frac{w}{2}\biggr\rfloor + C
\biggr)^{{C}/{2} }
\\
&&\qquad \leq\biggl\lfloor\frac{w}{2}\biggr\rfloor E2^X P\biggl(\operatorname{Po}(d)\geq w + \frac{C}{2} \biggr)
\nonumber
\\
&&\qquad \leq\frac1{100} P\bigl(\operatorname{Po}(d)\geq
w\bigr),\nonumber
\end{eqnarray}
where line 4 follows from the fact that $\frac{C}{2} \geq3$ and equation
(\ref{eGWExposureAssum2}), and line 5 follows from equation (\ref
{eGWExposureAssum1}). Finally,
%
\begin{eqnarray}
&& \sum_{j=\lfloor{w}/{2}\rfloor +1}^{w+{C}/{2} } \sum
_{x=j}^{w+{C}/{2} } 2^x P\bigl(
\operatorname{Po}(d)\geq w-j+1+C\bigr)^{j} P(X=x)
\nonumber
\\
&&\qquad \leq\sum_{j=\lfloor{w}/{2}\rfloor +1}^{w+{C}/{2} } \sum
_{x=j}^{w+{C}/{2} } 2^{w+{C}/{2} } P\bigl(
\operatorname{Po}(d)\geq w-j+1+C\bigr)
^{\lfloor{w}/{2}\rfloor +1} \nonumber\\
&&\qquad\quad\hspace*{70.2pt}{}\times P\bigl(\operatorname{Po}(d)
=x\bigr)
\nonumber\\[-8pt]\\[-8pt]
&&\qquad \leq\sum_{j=\lfloor{w}/{2}\rfloor +1}^{w+{C}/{2} } \sum
_{x=j}^{w+{C}/{2} } 2^{w+{C}/{2} } P\biggl(
\operatorname{Po}(d)\geq\frac{C}{2} \biggr)^{\lfloor{w}/{2}\rfloor }
\nonumber
\\
&&\hspace*{70.2pt}\qquad\quad{} \times P\bigl(\operatorname{Po}(d) = w-x+C\bigr)
\nonumber\\
&&\hspace*{70.2pt}\qquad\quad{} \times  P\bigl(\operatorname{Po}(d) =x
\bigr),\nonumber
\end{eqnarray}
where the second line follows since $x\leq w+\frac{C}{2} $ and $j
\geq\lfloor\frac{w}{2}\rfloor +1$, and the third line follows from the
fact that $w-j+1+C$ is greater than both $\Cdu $ and $w-x+C+1$, and
applying equation (\ref{eGWExposureAssum0}) which
says that $P(\operatorname{Po}(d) = w-x+C) \geq
P(\operatorname{Po}(d) \geq w-x+C +1)$. Then
%
\begin{eqnarray}
\label{eGWExposureE}
&&
\sum_{j=\lfloor{w}/{2}\rfloor +1}^{w+{C}/{2} }
\sum_{x=j}^{w+{C}/{2} } 2^{w+{C}/{2} } P\biggl(
\operatorname{Po}(d)\geq\frac{C}{2} \biggr)^{\lfloor{w}/{2}\rfloor } P\bigl(\operatorname{Po}(d) =
w-x+C\bigr) \nonumber\hspace*{-15pt}\\
&&\hspace*{70.7pt}{}\times P\bigl(\operatorname{Po}(d) =x\bigr)
\nonumber\hspace*{-15pt}
\\
&&\qquad \leq\sum_{j=\lfloor{w}/{2}\rfloor +1}^{w+{C}/{2} } \sum
_{x=j}^{w+{C}/{2} } 2^{w+{C}/{2} } P\biggl(
\operatorname{Po}(d)\geq\frac{C}{2} \biggr)^{\lfloor{w}/{2}\rfloor } 2^{w+C} P\bigl(
\operatorname{Po}(d) =w+C\bigr)\hspace*{-15pt}
\\
&&\qquad \leq \biggl(w + \frac{C}{2} \biggr)^2 2^{2w+{3C}/{2}} P\biggl( \operatorname{Po}(d)
\geq\frac{C}{2} \biggr)^{\lfloor{w}/{2}\rfloor } P\bigl(\operatorname{Po}(d) =w+C\bigr)
\nonumber\hspace*{-15pt}
\\
&&\qquad \leq\frac1{100} P\bigl(\operatorname{Po}(d)\geq
w\bigr),\nonumber\hspace*{-15pt}
\end{eqnarray}
where the second line follows from equation (\ref
{eGWExposureExpSum}), and the final line follows from equation (\ref
{eGWExposureAssum4}). Combining equations (\ref{eGWExposureA})
through (\ref{eGWExposureE}), we have that for $w\geq2$,
\[
P(W \geq w+C) \leq\tfrac1{25}P \bigl(\operatorname{Po}(d)\geq w \bigr) \leq P
\bigl(\operatorname{Po}(d)\geq w \bigr).
\]
Combining this with equation (\ref{eGWExposureAssum5}) completes
the proof.
\end{pf}

We now prove Lemma~\ref{lcutwidthGW}.

\begin{pf*}{Proof of Lemma~\ref{lcutwidthGW}}
Take $C'=C+1$ where $C$ is the constant from Lem\-ma~\ref
{lpoissonIdentity}. We prove the result by induction on $\ell$. When
$\ell=0$ a 0 level Galton--Watson branching process tree is just a
single vertex which has cut-width 0, so the statement is trivially
satisfied. When $\ell\geq1$, the subtrees attached to the root are
independent $\ell-1$ level Galton--Watson branching process trees, so
by the inductive hypothesis, Lemmas~\ref{lcutwidthTreeClaim} and
\ref{lpoissonIdentity}, we have that $\mathcal{E}(T)$ is
stochastically dominated by the distribution $C'\ell+ \operatorname{Po}(d)$.
\end{pf*}

\section{Spatial mixing}



\subsection{SAW trees}

Weitz \cite{Weitz06} developed the tree of self-avoiding walks
construction, which enables the calculation of marginal distributions
of a Gibbs measure on a graph by calculating marginal distributions on
a specially constructed tree. This construction, along with the
censoring inequality, will be a major tool in our proof. For a graph
$G$ and a vertex $v$, we denote the tree of self-avoiding paths from
$V$ in $G$ as $T_{\mathrm{saw}}(G,v)$. This is the tree of paths in
$G$ starting from $v$ and not intersecting
themselves, except possibly at the terminal vertex of the path. Through this
construction each vertex in $T_{\mathrm{saw}}(G,v)$ can be mapped
to a vertex in $G$. This gives a natural way to relate
a subset $\Lambda\subset V$ as the pullback of this map which denote
$\varphi(\Lambda)\subset T_{\mathrm{saw}}(G,v)$. We extend this to relating
configurations $\eta_\Lambda$ to the
corresponding configurations $\eta_{\varphi(\Lambda)}$ on $\varphi
(\Lambda)$. Furthermore if $A,B\subset V$ then
$d(A,B)=d(\varphi(A), \varphi(B))$. Each vertex (edge) of $T_{\mathrm{saw}}$
maps to a
vertex (edge) in $G$ so $P_{T_{\mathrm{saw}}}$ is defined by taking the
corresponding external field and interactions. Then Theorem 3.1 of
\cite{Weitz06} gives the
following result.

\begin{lemma}[(Weitz \cite{Weitz06})]\label{lemWeitz}
For a graph $G$ and $v\in G$, there exists $A\subset T_{\mathrm{saw}}$ and
a configuration $\nu_A$ on $A$ such that for any $\Lambda\subset V$
and configuration $\eta_\Lambda$ on $\Lambda$, such that
\[
P_G(\sigma_v=+\mid\sigma_\Lambda)=P_{T_{\mathrm{saw}}}(
\sigma_v=+\mid\sigma_{\varphi(\Lambda)\setminus A}=\eta_{\varphi(\Lambda)\setminus
A},
\sigma_A=\nu_{A}).
\]
The set $A$ is the set of leaves in $T_{\mathrm{saw}}$ corresponding to the
terminal vertices of paths which
return to a vertex already visited by the path. The construction of
$\nu_A$ is described in \cite{Weitz06}.
\end{lemma}

\subsection{Spatial correlations on trees}

We consider the effect that conditioning the vertices of a tree has on
the marginal distribution of the spin at the root. It will be
convenient to compare this probability to the Ising model with the same
interaction strengths $\beta_{uv}$ but no external field ($h\equiv0$)
which we will denote $\tilde{P}$.

\begin{lemma}\label{lemtreeboundaryeffect}
Suppose that $T$ is a tree, $P$ is the Ising model with arbitrary
external field
(including $h_u = \pm\infty$ meaning that $\sigma_u$ is set to $\pm
$) and $0\leq\beta_{u,v} \leq\beta$ for all $(u,v)\in E$. Let
$U\subseteq\Lambda\subset V$, and let $\eta^+,\eta^-$ be two
configurations on $\Lambda$ which differ only on $U$ with $\eta_U^+\equiv+,\eta_U^-\equiv-$.
Then for all $v\in V$,
\[
0\leq P\bigl(\sigma_v=+\mid\sigma_{\Lambda}=\eta^+\bigr)- P
\bigl(\sigma_v=+\mid\sigma_{\Lambda}=\eta^-\bigr) \leq\sum
_{u\in U} (\tanh \beta)^{d(u,v)}.
\]
\end{lemma}

\begin{pf}
The inequality
\[
0\leq P\bigl(\sigma_v=+\mid\sigma_{\Lambda}=\eta^+\bigr)- P
\bigl(\sigma_v=+\mid\sigma_{\Lambda}=\eta^-\bigr)
\]
simply follows from the monotonicity of the ferromagnetic Ising model.
Now suppose that the set $U$ is a single vertex $u$.
Lemma 4.1
of \cite{BeKeMoPe05} implies that for any vertices $v,u\in T$,
%
\begin{eqnarray}
\label{enoFieldEffect}
&&
P(\sigma_v=+\mid\sigma_{u}=+)- P(
\sigma_v=+\mid\sigma_{u}=-)
\nonumber\\[-8pt]\\[-8pt]
&&\qquad \leq\tilde{P}(\sigma_v=+\mid\sigma_{u}=+)-
\tilde{P}(\sigma_v=+\mid\sigma_{u}=-).\nonumber
\end{eqnarray}
If $u_0,u_1,\ldots,u_l$ are a path of vertices in $T$, then a simple
calculation yields that
%
\begin{eqnarray}
\label{epathDecay}
\tilde{P}(\sigma_{u_k}=+\mid\sigma_{u_0}=+)-
\tilde{P}(\sigma_{u_k}=+\mid\sigma_{u_0}=-)
&=& \prod
_{i=1}^k \tanh \beta_{u_{i-1} u_i} \nonumber\\[-8pt]\\[-8pt]
&\leq&(\tanh \beta)^k.\nonumber
\end{eqnarray}
Conditioning is equivalent to setting an infinite external field, so
equations (\ref{enoFieldEffect}) and (\ref{epathDecay}) imply that
%
\begin{equation}
\label{esingleSpinFlip} P\bigl(\sigma_v=+\mid\sigma_{\Lambda}=
\eta^+\bigr)- P\bigl(\sigma_v=+\mid \sigma_{\Lambda}=\eta^-\bigr)
\leq(\tanh\beta)^{d(u,v)}.
\end{equation}
We now consider a general $U$. Let $u_1,\ldots,u_{|U|}$ be an
arbitrary labeling of the vertices of $U$. Take a sequence of
configurations $\eta^0, \eta^1,\ldots, \eta^{|U|}$ on $\Lambda$
with $\eta^0=\eta^-$ and $ \eta^{|U|} =\eta^+$ where consecutive
configurations $\eta^{i-1}$ and $\eta^i$ differ only at $u_i$ with
$\eta^i_{u_i}=+$ and $\eta^{i-1}_{u_i}=-$. By equation (\ref
{esingleSpinFlip}) we have that
\[
P\bigl(\sigma_v=+\mid\sigma_{\Lambda}= \eta^{i+1}
\bigr)- P\bigl(\sigma_v=+\mid\sigma_{\lambda}= \eta^i
\bigr) \leq(\tanh \beta)^{d(v,u_i)}
\]
and so
\[
P\bigl(\sigma_v=+\mid\sigma_{\Lambda}=\eta^+\bigr)- P\bigl(
\sigma_v=+\mid\sigma_{\Lambda}=\eta^-\bigr) \leq\sum
_{u\in U} (\tanh \beta)^{d(u,v)},
\]
which completes the proof.
\end{pf}

\subsection{Continuous time to discrete time}

\begin{lemma}\label{lctsToDiscrete}
Suppose that in continuous time starting from the all $+$ and all $-$
configurations the Gibbs sampler under the monotone coupling couples
with probability at least $\frac78$ by time $T\geq1$. Then the Gibbs
sampler in discrete time under the monotone coupling couples with
probability at least $1-\frac1{2e}$ by time $\lceil5Tn \rceil$ and
hence has mixing time at most $\lceil5Tn\rceil$.
\end{lemma}
\begin{pf}
Let $M$ denote the number of updates of the continuous dynamics up to
time $T$. Then $M$ is distributed as a Poisson random variable with
mean $Tn$. For some integer $m$, the final state of the continuous time
Gibbs sampler conditioned on $M=m$ is the same as the final state of
the discrete Gibbs sampler with $m$ steps. So the probability of
coupling in the discrete time after $m$ steps is at least $\frac78 -
P(\operatorname{Po}(Tn)>m)$. So if $m\geq5Tn$, then by Markov's theorem,
\[
P\bigl(\operatorname{Po}(Tn)>m\bigr) \leq\frac{Ee^{
\operatorname{Po}(Tn)}}{e^{5Tn}}=e^{Tn(e-1)-5Tn}
\leq e^{-3}.
\]
Since $\frac78 -e^{-3} > 1 - \frac1{2e}$, the discrete chain couples
by time $5Tn$ with probability at least $1-\frac1{2e}$. Hence the
mixing time is at most $\lceil5Tn\rceil$.
\end{pf}

\subsection{\texorpdfstring{Proof of Lemma \protect\ref{lemmainSM}}{Proof of Lemma 3}}

We now prove Lemma~\ref{lemmainSM} by applying Lemmas~\ref{lemWeitz}
and~\ref{lemtreeboundaryeffect} to a small graph centered at $v$.


\begin{pf*}{Proof of Lemma~\ref{lemmainSM}}
Let $T$ denote the tree of self avoiding walks on $G$ from~$v$,
$T_{\mathrm{saw}}(G,v)$. Let $\varphi(S(v,R))$ denote the vertices in $T$ which
correspond to vertices in $S(v,R)$, and for each $u\in S(v,R)$ let
$\varphi(u)$ denote the set of vertices in $T$ which correspond to
$u$. Then by Lemmas~\ref{lemWeitz} and~\ref{lemtreeboundaryeffect},
%
\begin{eqnarray}
\label{eauBound} a_u&=& \sup_{\eta^+,\eta^-}
P_{T_{\mathrm{saw}}}\bigl(\sigma_v=+\mid\sigma_{\varphi
(\Lambda)\setminus A}=
\eta^+_{\phi(\Lambda)\setminus A},\sigma_A=\nu_{A}\bigr)
\nonumber
\\
&&\hspace*{22pt}{} - P_{T_{\mathrm{saw}}}\bigl(\sigma_v=+\mid\sigma_{\varphi(\Lambda
)\setminus A}=
\eta^-_{\phi(\Lambda)\setminus A},\sigma_A=\nu_{A}\bigr)
\\
& \leq &\sum_{w\in\varphi(u) } \tanh^{d(v,w)} \beta.\nonumber
\end{eqnarray}
Applying this bound,
\begin{eqnarray*}
\sum_{u\in S(v,R)} a_u &\leq& \sum
_{u\in S(v,R)} \sum_{w\in
\varphi(u) }
\tanh^{d(v,w)} \beta
\\
&=& \sum_{w \in\varphi(S(v,R))} \tanh^{d(v,w)} \beta
\\
&\leq& \sum_{w \in T\dvtx d(w,v)\geq R} \tanh^{d(v,w)} \beta,
\end{eqnarray*}
where the final inequality follows from the fact that $d(v,\varphi
(S(v,R)))\geq m$. Now since $T$ has maximum degree $d$ for each $\ell
$, there are at most $d(d-1)^{\ell-1}$ vertices at distance $\ell$
from $v$. It follows that
\begin{eqnarray*}
\sum_{u\in S(v,R)} a_u &\leq& \sum
_{w \in T\dvtx d(w,v)\geq R} \tanh^{d(v,w)} \beta
\\
&\leq& \sum_{\ell=R}^\infty d(d-1)^{\ell-1}
\tanh^{\ell} \beta
\\
&=& \frac{d(d-1)^{R-1} \tanh^{R} \beta}{1-(d-1)\tanh\beta}
\\
&\leq& \frac14
\end{eqnarray*}
as required.
\end{pf*}



\subsection{\texorpdfstring{Proof of Lemma \protect\ref{lemERSM}}{Proof of Lemma 6}}

We now prove Lemma~\ref{lemERSM}.



\begin{pf*}{Proof of Lemma~\ref{lemERSM}}
We need to establish the spatial mixing condition. Recall that
\[
a_u=\sup_{\eta^+,\eta^-} P\bigl(\sigma_v=+\mid
\sigma_{\Lambda}=\eta^+\bigr)- P\bigl(\sigma_v=+\mid
\sigma_{\Lambda}=\eta^-\bigr)
\]
and by equation (\ref{eauBound}),
\[
a_u \leq\sum_{w\in\varphi(u) }
\tanh^{d(v,w)} \beta.
\]
Now $t(v,R)\leq1$ with high probability for all $v\in V$ by Lemma \ref
{lERGraphProperties}, so $B(v,R)$ is a tree or unicyclic. Hence every
$u\in S(v,R)$ appears, at most, twice in the tree of self-avoiding
walks, which gives $|\varphi(u)|\leq2$ and $d(v,\varphi(u))=R$. Thus
for all $v\in V$ with high probability,
\begin{eqnarray*}
\sum_{u\in S(v,R)} a_u &\leq& \sum
_{u\in S(v,R)} \sum_{w\in\varphi
(u) }
\tanh^{d(v,w)} \beta
\\
&\leq& 2\mathfrak{X} \tanh^{R} \beta
\\
&=& 6\bigl(1-d^{-1}\bigr)^{-1}(d\tanh\beta)^{R}
\log n
\\
&=& o(1),
\end{eqnarray*}
which establishes the spatial mixing condition.
%
\end{pf*}

\section{Infinite graphs}\label{sinfinite}

Up to this point,we have only dealt with finite graphs; however, the
Ising model and the Glauber dynamics can be defined on infinite graphs
as well; see, for example, \cite{Liggett85}. The spatial mixing
property of uniqueness says that there is a unique Gibbs measure for
the interacting particle system; one formulation of this is that for\vadjust{\goodbreak}
every finite set $A\subset V$, we have that
\[
\limsup_{R\rightarrow\infty} \sup_{\eta,\eta'} \bigl\| P (\sigma_{A} =\cdot
\mid\sigma_{S(A,R)}=\eta ) - P \bigl(\sigma_{A} =\cdot\mid
\sigma_{S(A,R)}=\eta' \bigr) \bigr\|_{\mathrm{TV}} = 0,
\]
where $S(A,R)=\{u\in V\dvtx d(u,A)=R\}$, and $\eta,\eta'$ are
configurations on $S(A,R)$. This says that the configuration on $A$ is
asymptotically independent of the spins a large distance away.
In the context of the ferromagnetic Ising model this is equivalent to
%
\begin{equation}
\label{euniquenessCriterion} P ( \sigma_v = + \mid
\sigma_{S(v,R)} \equiv+ ) - P ( \sigma_v = + \mid
\sigma_{S(v,R)} \equiv- ) \longrightarrow0
\end{equation}
for all $v\in V$ as $R\rightarrow\infty$.
Combining Lemmas~\ref{lemWeitz} and~\ref{lemtreeboundaryeffect}
it follows that condition (\ref{emainCondition}) implies uniqueness.
This was also noted in \cite{ZhaLiaBai11}.

The following lemma shows that given uniqueness the Glauber dynamics on
an infinite graph can locally be approximated by the Glauber dynamics
of the Ising model on finite graphs. For a fixed finite set $U\subset
$, let $\sigma^{*\ell}$ denote a random configuration according to
the stationary distribution of the Ising model on the induced subgraph
$G_\ell$ whose vertex set is given by $U_\ell:=\{u\in V\dvtx d(u,U)\leq
\ell\}$. Let $\sigma^{*\ell}(t)$ denote the Glauber dynamics of this
Ising model started from the stationary distribution.

\begin{lemma}\label{llocalJointApprox}
Let $G$ be an infinite graph with maximum degree $d$, and suppose for
some $\{ \beta_{(u,v)} \}$ and $\{ h_u \}$ that the Ising model has
the uniqueness property, and let $U$ be a finite subset of $V$. With
$\sigma^{*\ell}_U(t)$ defined as above,
\[
\bigl(\sigma^{*\ell}_U(0),\sigma^{*\ell}_U(1)
\bigr) \rightarrow \bigl(\sigma_U(0),\sigma_U(1)
\bigr)
\]
jointly in distribution as $\ell\rightarrow\infty$.
\end{lemma}
\begin{pf}
Fix an $\varepsilon>0$. It is sufficient to show that for some $\ell'$
we can couple $ (\sigma^{*\ell}_U(0),\sigma^{*\ell}_U(1)
)$ and $  (\sigma_U(0),\sigma_U(1) )$ with probability at
least $1-\varepsilon$ when $\ell>\ell'$.
Fix some positive integer $m$ large enough so that
\[
P\bigl(\operatorname{Poisson}(1) \geq m\bigr)<\tfrac12\varepsilon d^{-m}
|U|^{-1}.
\]
By the uniqueness property as $\ell\rightarrow\infty$, we have that
$\sigma^{*\ell}_{U_m}$ converges in distribution to $\sigma_{U_m}$.
So for some $\ell'$ when $\ell>\ell'$, we can couple initial
configurations $\sigma^{*\ell}(0)$ and $\sigma(0)$ so that $\sigma^{*\ell}_{U_m}(0)$ and $\sigma_{U_m}(0)$ agree with probability at
least $1-\varepsilon/2$. Now couple the Glauber dynamics by using the
same sequence of updates for each chain within $U_\ell$.

We now bound the probability that there is disagreement between $\sigma^{*\ell}_{U}(1)$ and $\sigma_{U}(1)$, given that $\sigma^{*\ell
}_{U_m}(0)$ and $\sigma_{U_m}(0)$ agree. We will call a sequence
$u_1,\ldots, u_k$ of vertices a \textit{path} if $u_i$ and $u_{i+1}$
are adjacent for each $i$. An update can only create a disagreement at
the vertex if a neighboring vertex already has a disagreement. Hence a
vertex $u$ can only have a disagreement by time $t$ if there is a path
of vertices from $u_1, \ldots,u_k=u$ such that the vertices in the
path are updated by the Glauber dynamics in that order before time 1
and $u_1 \in U_m\setminus U_{m-1}$.

Hence the event $\sigma^{*\ell}_{U}(1)\neq\sigma_{U}(1)$ is
dominated by the event that there is a path of updates of vertices
$u_1,\ldots,u_m$, updated in that order before time $1$ with $u_m\in
U$. For each fixed path the probability that those vertices are updated
in that order is $P(\operatorname{Poisson}(1) \geq m)$. There are at most $d^m
|U|$ such paths of vertices, so by a union bound and our choice of $m$,
the probability of a disagreement reaching $|U|$ is at most $\varepsilon
/2$. It follows that we can couple $ (\sigma^{*\ell}_U(0),\sigma^{*\ell}_U(1) )$ and $  (\sigma_U(0),\sigma_U(1) )$
with probability at least $1-\varepsilon$, which completes the proof.
\end{pf}

We now show how the spectral gap bounds for the finite graph dynamics
imply spectral gap bounds for infinite graph dynamics. The following
lemma completes Theorem~\ref{tmain}.

\begin{lemma}
Let $G$ be a infinite graph with maximum degree $d$, and suppose for
some $\{ \beta_{(u,v)} \}$ and $\{ h_u \}$ the Ising model has the
uniqueness property. Further suppose that
for every finite subgraph $G'$ of $G$, the Ising model on $G'$ has
continuous time spectral gap bounded below by $\lambda^*$. Then the
infinite volume dynamics has spectral gap bounded below by $\lambda^*$.
\end{lemma}
\begin{pf}
First we may assume that the graph is connected since the spectral gap
is the minimum of the spectral gaps of the dynamics projected onto
individual components. We will use the characterization of the spectral
gap that
\[
\mathrm{Gap} = -\mathop{\log\sup}_f \frac{\Cov(  ( f(\sigma
(0)),f(\sigma(1)) )}{\Var f(\sigma(0))},
\]
where the supremum is over all square integrable functions $f\dvtx \{+,-\}^V\to\mathbb{R}$ with $E f =0$. Fix a vertex $v$, and for such a
function $f$, we define the bounded function $f_R\dvtx\{+,-\}^{B(v,R)}\to
\mathbb{R}$ by
\[
f_R(\sigma) = E \bigl( f(\sigma) \mid\sigma_{B(v,R)} \bigr).
\]
Since every vertex is ultimately in $B(v,R)$ for $R$ sufficiently
large, by the $L^2$ martingale convergence theorem, $f_R(\sigma)$
converges to $f(\sigma)$ in $L^2$, and so
%
\begin{equation}
\label{elocalApproximation} \lim_{R\to\infty} \frac{\Cov(  ( f_R(\sigma(0)),f_R(\sigma
(1)) )}{\Var f_R(\sigma(0))} =
\frac{\Cov(  ( f(\sigma
(0)),f(\sigma(1)) )}{\Var f(\sigma(0))}.
\end{equation}
In particular, this means that in the supremum, we only need consider
bounded functions which are determined by a finite number of spins. So
suppose that $g$ is such a bounded function depending only on $\sigma_U$ for some finite $U\subset V$.

By Lemma~\ref{llocalJointApprox} we have that $ (\sigma^{*\ell
}_U(0),\sigma^{*\ell}_U(1) )$ converges jointly in distribution
to $ (\sigma_U(0),\sigma_U(1) )$. Hence using our
assumption on the spectral gap on finite subgraphs, we have that
\[
\lambda^* \leq\lim_{\ell\to\infty} {-}\log\frac{\Cov(  (
g(\sigma^{*\ell}(0)),g(\sigma^{*\ell}(1)) )}{\Var g(\sigma^{*\ell}(0))} = -\log
\frac{\Cov(  ( g(\sigma(0)),g(\sigma
(1)) )}{\Var g(\sigma(0))},
\]
which establishes $\lambda^*$ as a lower bound on the spectral gap.
\end{pf}

\section{Conclusion}

The proof of Theorem~\ref{tgeneral} naturally extends to more general
monotone systems. Moreover, instead of censoring outside a ball of
radius $R$ about a vertex $v$, we could instead look at the general,
well-chosen sets $v\in W_v\subset V$. We let $S_v$ denote the boundary
set $\{u\in V\setminus W_v\dvtx d(u,W_v)=1\}$. We consider the following setup.
There is a spin set $\Omega$ which is ordered with a maximal element
$+$ and a minimal element $-$.
The order on $\Omega$ naturally extends to a partial order on $\Omega^V$ where $V$ is the vertex set of a graph by letting $\sigma_1 \leq
\sigma_2$ if and only if $\sigma_1(v) \leq\sigma_2(v)$ for all $v
\in V$. A measure
$P$ on $\Omega^V$ is called monotone if for all $v \in V$ and all $a
\in\Omega$,
\[
P\bigl[\sigma(v) \geq a \mid \sigma(w\dvtx w \neq v) = \sigma_1\bigr] \geq
P\bigl[\sigma (v) \geq a \mid \sigma(w\dvtx w \neq v) = \sigma_2\bigr],
\]
whenever $\sigma_1 \geq\sigma_2$. We may now state a generalization
of Theorem~\ref{tgeneral}.

\begin{maintheorem}\label{tveryGeneral}
Let $G$ be a graph on $n\geq2$ vertices, and let $P(\sigma)$ be any
monotone Gibbs measure on $G$.

Suppose that there exist constants $T,\mathfrak{X}\geq1$ and for each
$v\in V$ there is a subset $W_v\subset V$ containing $v$ such that the following
three conditions hold:
\begin{itemize}
\item\textup{Volume}: The volume of $W_v$ satisfies $|W_v|\leq
\mathfrak{X}$.

\item\textup{Local mixing}: For any configuration $\eta$ on $S_v$,
the continuous time mixing time of the Gibbs sampler on $W_v$ with
fixed boundary condition $\eta$ is bounded above by $T$.

\item\textup{Spatial mixing}: For each vertex $u\in S_v$, define
%
\begin{equation}
\label{eqSM1gen} a_u=\sup_{\eta^+,\eta^-} d_{\mathrm{TV}}
\bigl(P\bigl(\sigma_v=\cdot\mid\sigma_{\Lambda}=
\eta^1\bigr), P\bigl(\sigma_v=\cdot\mid
\sigma_{\Lambda}=\eta^2\bigr) \bigr),
\end{equation}
where the supremum is over configurations $\eta^1,\eta^2$ on $S_v$
which differ only at $u$. Then
%
\begin{equation}
\label{eqSMgen} \sum_{u\in S_v} a_u \leq
\frac14.
\end{equation}
\end{itemize}
Then starting from the all $+$ and all $-$ configurations in continuous
time, the monotone coupling couples with probability at least $\frac
78$ by time $T \lceil\log8\mathfrak{X} \rceil ( 3 + \log_2 n
)$.

It follows that the mixing time of the Gibbs sampler in
continuous time satisfies
\[
\tau_{\mathrm{mix}} \leq T \lceil\log8\mathfrak{X} \rceil ( 3 +
\log_2 n ).
\]
\end{maintheorem}

While Theorem~\ref{tveryGeneral} applies to general monotone systems,
the use of the censoring lemma of Peres and Winkler does not allow us
to extend it to nonmonotone systems such as random colorings. A major
open problem is how to relate spatial mixing to temporal mixing in
nonmonotone settings, for example, for the hardcore model, the
antiferromagnetic Ising model or the coloring model.

\subsection{Open problems}
We showed that condition (\ref{emainCondition}) establishes a uniform
lower bound on the spectral gap of the continuous time dynamics over
all graphs. It would be of interest to establish whether or not this is
also true for bounds on the Log-Sobolev constant as well.

As discussed in the \hyperref[intro]{Introduction}, our results give rise to the
following conjecture concerning nonmonotone systems.

\begin{conjecture}
The Gibbs sampler for the antiferromagnetic Ising model (with no
external field) is rapidly mixing on any graph whose maximum degree
$d$, for any inverse temperature $\beta$ below the uniqueness
threshold for the Ising model on the $d$-regular tree.

Similarly, the Gibbs sampler for the hardcore model is rapidly mixing
on any graph whose maximum degree is $d$ for any fugacity $\lam$ below
the uniqueness threshold for the hard-core model on the $d$-regular tree.
\end{conjecture}
We recall that for both of these models, the mixing time on almost all
random $d$-regular bipartite graphs is exponential in $n$ the size of
the graph beyond the uniqueness threshold \cite
{MoWeWo09,GerschenfeldMontanari07,DemboMontanari09}, so our
conjecture is that uniqueness on the tree exactly corresponds to rapid
mixing of the Gibbs sampler. A similar conjecture can be made with
respect to the coloring model.

\section*{Acknowledgments}

Most of this work was carried out while A. Sly was a student at UC Berkeley
and on visits to Rome Tre and the Weizmann Institute. A.~Sly would like to
thank Fabio Martinelli for useful discussions.



\printaddresses


\begin{thebibliography}{32}

\bibitem{AldousFillu}
\begin{bmisc}[auto:STB|2012/07/16|09:13:59]
\bauthor{\bsnm{Aldous},~\bfnm{D.}\binits{D.}} \AND
  \bauthor{\bsnm{Fill},~\bfnm{J.~A.}\binits{J.~A.}}
\bhowpublished{Reversible Markov chains and random walks on graphs. Unpublished
  manuscript. Available at
  \texttt{\href{http://stat-www.berkeley.edu/users/aldous/book.html}{http://stat-www.berkeley.edu/users/}
  \href{http://stat-www.berkeley.edu/users/aldous/book.html}{aldous/book.html}}.}
\bptok{imsref}%
\end{bmisc}
\endbibitem

\bibitem{BeKeMoPe05}
\begin{barticle}[auto:STB|2012/07/16|09:13:59]
\bauthor{\bsnm{Berger},~\bfnm{N.}\binits{N.}},
  \bauthor{\bsnm{Kenyon},~\bfnm{C.}\binits{C.}},
  \bauthor{\bsnm{Mossel},~\bfnm{E.}\binits{E.}} \AND
  \bauthor{\bsnm{Peres},~\bfnm{Y.}\binits{Y.}}
(\byear{2005}).
\btitle{Glauber dynamics on prob and hyperbolic graphs}.
\bjournal{Probab. Theory Related Fields}
\bvolume{131}
\bpages{311--340}.
\bptok{imsref}%
\end{barticle}
\endbibitem

\bibitem{Cesi01}
\begin{barticle}[mr]
\bauthor{\bsnm{Cesi},~\bfnm{Filippo}\binits{F.}}
(\byear{2001}).
\btitle{Quasi-factorization of the entropy and logarithmic {S}obolev
  inequalities for {G}ibbs random fields}.
\bjournal{Probab. Theory Related Fields}
\bvolume{120}
\bpages{569--584}.
\bid{doi={10.1007/PL00008792}, issn={0178-8051}, mr={1853483}}
\bptok{imsref}%
\end{barticle}
\endbibitem

\bibitem{DemboMontanari09}
\begin{barticle}[mr]
\bauthor{\bsnm{Dembo},~\bfnm{Amir}\binits{A.}} \AND
  \bauthor{\bsnm{Montanari},~\bfnm{Andrea}\binits{A.}}
(\byear{2010}).
\btitle{Ising models on locally tree-like graphs}.
\bjournal{Ann. Appl. Probab.}
\bvolume{20}
\bpages{565--592}.
\bid{doi={10.1214/09-AAP627}, issn={1050-5164}, mr={2650042}}
\bptok{imsref}%
\end{barticle}
\endbibitem

\bibitem{DobrushinShlosman85}
\begin{bincollection}[auto:STB|2012/07/16|09:13:59]
\bauthor{\bsnm{Dobrushin},~\bfnm{R.~L.}\binits{R.~L.}} \AND
  \bauthor{\bsnm{Shlosman},~\bfnm{S.~B.}\binits{S.~B.}}
(\byear{1985}).
\btitle{Constructive criterion for uniqueness of a Gibbs field}.
In \bbooktitle{Statistical Mechanics and Dynamical Systems, Volume 10}
(\beditor{\bfnm{J.}\binits{J.}~\bsnm{Fritz}},
  \beditor{\bfnm{A.}\binits{A.}~\bsnm{Jaffe}} \AND
  \beditor{\bfnm{D.}\binits{D.}~\bsnm{Szasz}}, eds.)
\bpages{347--370}.
\bptok{imsref}%
\end{bincollection}
\endbibitem

\bibitem{DSVW04}
\begin{barticle}[auto:STB|2012/07/16|09:13:59]
\bauthor{\bsnm{Dyer},~\bfnm{Martin}\binits{M.}},
  \bauthor{\bsnm{Sinclair},~\bfnm{Alistair}\binits{A.}},
  \bauthor{\bsnm{Vigoda},~\bfnm{Eric}\binits{E.}} \AND
  \bauthor{\bsnm{Weitz},~\bfnm{Dror}\binits{D.}}
(\byear{2004}).
\btitle{Mixing in time and space for lattice spin systems: A combinatorial
  view}.
\bjournal{Random Structures Algorithms}
\bvolume{24}
\bpages{461--479}.
\bptok{imsref}%
\end{barticle}
\endbibitem

\bibitem{GerschenfeldMontanari07}
\begin{bincollection}[auto:STB|2012/07/16|09:13:59]
\bauthor{\bsnm{Gershchenfeld},~\bfnm{A.}\binits{A.}} \AND
  \bauthor{\bsnm{Montanari},~\bfnm{A.}\binits{A.}}
(\byear{2007}).
\btitle{Reconstruction for models on random graphs}.
In \bbooktitle{Annual IEEE Symposium on Foundations of Computer Science}
\bpages{194--204}.
\bpublisher{IEEE Comput. Soc.}, \baddress{Los Alamitos, CA}.
\bptok{imsref}%
\end{bincollection}
\endbibitem

\bibitem{HayesSinclair05}
\begin{barticle}[mr]
\bauthor{\bsnm{Hayes},~\bfnm{Thomas~P.}\binits{T.~P.}} \AND
  \bauthor{\bsnm{Sinclair},~\bfnm{Alistair}\binits{A.}}
(\byear{2007}).
\btitle{A general lower bound for mixing of single-site dynamics on graphs}.
\bjournal{Ann. Appl. Probab.}
\bvolume{17}
\bpages{931--952}.
\bid{doi={10.1214/105051607000000104}, issn={1050-5164}, mr={2326236}}
\bptok{imsref}%
\end{barticle}
\endbibitem

\bibitem{HayesVigoda05}
\begin{barticle}[mr]
\bauthor{\bsnm{Hayes},~\bfnm{Thomas~P.}\binits{T.~P.}} \AND
  \bauthor{\bsnm{Vigoda},~\bfnm{Eric}\binits{E.}}
(\byear{2006}).
\btitle{Coupling with the stationary distribution and improved sampling for
  colorings and independent sets}.
\bjournal{Ann. Appl. Probab.}
\bvolume{16}
\bpages{1297--1318}.
\bid{doi={10.1214/105051606000000330}, issn={1050-5164}, mr={2260064}}
\bptok{imsref}%
\end{barticle}
\endbibitem

\bibitem{Higuchi93}
\begin{barticle}[auto:STB|2012/07/16|09:13:59]
\bauthor{\bsnm{Higuchi},~\bfnm{Y.}\binits{Y.}}
(\byear{1993}).
\btitle{Coexistence of infinite (*)-clusters II. Ising percolation in two
  dimensions}.
\bjournal{Probab. Theory Related Fields}
\bvolume{97}
\bpages{1--33}.
\bptok{imsref}%
\end{barticle}
\endbibitem

\bibitem{JerrumSinclair89}
\begin{barticle}[auto:STB|2012/07/16|09:13:59]
\bauthor{\bsnm{Jerrum},~\bfnm{M.}\binits{M.}} \AND
  \bauthor{\bsnm{Sinclair},~\bfnm{A.}\binits{A.}}
(\byear{1989}).
\btitle{Approximating the permanent}.
\bjournal{SIAM J. Comput.}
\bvolume{18}
\bpages{1149--1178}.
\bptok{imsref}%
\end{barticle}
\endbibitem

\bibitem{KeMoPe01}
\begin{bincollection}[auto:STB|2012/07/16|09:13:59]
\bauthor{\bsnm{Kenyon},~\bfnm{C.}\binits{C.}},
  \bauthor{\bsnm{Mossel},~\bfnm{E.}\binits{E.}} \AND
  \bauthor{\bsnm{Peres},~\bfnm{Y.}\binits{Y.}}
(\byear{2001}).
\btitle{Glauber dynamics on trees and hyperbolic graphs}.
In \bbooktitle{42nd IEEE Symposium on Foundations of Computer Science (Las
  Vegas, NV)}
\bpages{568--578}.
\bpublisher{IEEE Comput. Soc.}, \baddress{Los Alamitos, CA}.
\bptok{imsref}%
\end{bincollection}
\endbibitem

\bibitem{KMRSZ07}
\begin{barticle}[auto:STB|2012/07/16|09:13:59]
\bauthor{\bsnm{Kr{\c{z}}aka{\l}a},~\bfnm{F.}\binits{F.}},
  \bauthor{\bsnm{Montanari},~\bfnm{A.}\binits{A.}},
  \bauthor{\bsnm{Ricci-Tersenghi},~\bfnm{F.}\binits{F.}},
  \bauthor{\bsnm{Semerjian},~\bfnm{G.}\binits{G.}} \AND
  \bauthor{\bsnm{Zdeborov{\'a}},~\bfnm{L.}\binits{L.}}
(\byear{2007}).
\btitle{Gibbs states and the set of solutions of random constraint satisfaction
  problems}.
\bjournal{Proc. Natl. Acad. Sci. USA}
\bvolume{104}
\bpages{10318--10323}.
\bptok{imsref}%
\end{barticle}
\endbibitem

\bibitem{LevPerWil09}
\begin{bbook}[mr]
\bauthor{\bsnm{Levin},~\bfnm{David~A.}\binits{D.~A.}},
  \bauthor{\bsnm{Peres},~\bfnm{Yuval}\binits{Y.}} \AND
  \bauthor{\bsnm{Wilmer},~\bfnm{Elizabeth~L.}\binits{E.~L.}}
(\byear{2009}).
\btitle{Markov Chains and Mixing Times}.
\bpublisher{Amer. Math. Soc.}, \blocation{Providence, RI}.
\bnote{With a chapter by James G. Propp and David B. Wilson}.
\bid{mr={2466937}}
\bptok{imsref}%
\end{bbook}
\endbibitem

\bibitem{Liggett85}
\begin{bbook}[auto:STB|2012/07/16|09:13:59]
\bauthor{\bsnm{Liggett},~\bfnm{Thomas~M.}\binits{T.~M.}}
(\byear{1985}).
\btitle{Interacting Particle Systems}.
\bseries{Grundlehren der Mathematischen Wissenschaften [Fundamental Principles
  of Mathematical Sciences]}
\bvolume{276}.
\bpublisher{Springer}, \baddress{New York}.
\bptok{imsref}%
\end{bbook}
\endbibitem

\bibitem{Lyons89}
\begin{barticle}[auto:STB|2012/07/16|09:13:59]
\bauthor{\bsnm{Lyons},~\bfnm{R.}\binits{R.}}
(\byear{1989}).
\btitle{The Ising model and percolation on trees and tree-like graphs}.
\bjournal{Comm. Math. Phys.}
\bvolume{125}
\bpages{337--353}.
\bptok{imsref}%
\end{barticle}
\endbibitem

\bibitem{Martinelli99}
\begin{bincollection}[auto:STB|2012/07/16|09:13:59]
\bauthor{\bsnm{Martinelli},~\bfnm{F.}\binits{F.}}
(\byear{1999}).
\btitle{Lectures on Glauber dynamics for discrete spin models}.
In \bbooktitle{Lectures on Probability Theory and Statistics}.
\bseries{Lecture Notes in Math.}
\bvolume{1717}
\bpages{93--191}.
\bpublisher{Springer}, \baddress{Berlin}.
\bptok{imsref}%
\end{bincollection}
\endbibitem

\bibitem{MartinelliOliveri94a}
\begin{barticle}[auto:STB|2012/07/16|09:13:59]
\bauthor{\bsnm{Martinelli},~\bfnm{F.}\binits{F.}} \AND
  \bauthor{\bsnm{Olivieri},~\bfnm{E.}\binits{E.}}
(\byear{1994}).
\btitle{Approach to equilibrium of Glauber dynamics in the one phase region. I.
  The attractive case}.
\bjournal{Comm. Math. Phys.}
\bvolume{161}
\bpages{447--486}.
\bptok{imsref}%
\end{barticle}
\endbibitem

\bibitem{MartinelliOliveri94b}
\begin{barticle}[auto:STB|2012/07/16|09:13:59]
\bauthor{\bsnm{Martinelli},~\bfnm{F.}\binits{F.}} \AND
  \bauthor{\bsnm{Olivieri},~\bfnm{E.}\binits{E.}}
(\byear{1994}).
\btitle{Approach to equilibrium of Glauber dynamics in the one phase region.
  II. The general case}.
\bjournal{Comm. Math. Phys.}
\bvolume{161}
\bpages{487--514}.
\bptok{imsref}%
\end{barticle}
\endbibitem

\bibitem{MaSiWe03a}
\begin{bincollection}[auto:STB|2012/07/16|09:13:59]
\bauthor{\bsnm{Martinelli},~\bfnm{F.}\binits{F.}},
  \bauthor{\bsnm{Sinclair},~\bfnm{A.}\binits{A.}} \AND
  \bauthor{\bsnm{Weitz},~\bfnm{D.}\binits{D.}}
(\byear{2003}).
\btitle{The Ising model on trees: Boundary conditions and mixing time}.
In \bbooktitle{Proceedings of the Forty Fourth Annual Symposium on Foundations
  of Computer Science}
\bpages{628--639}.
\bptok{imsref}%
\end{bincollection}
\endbibitem

\bibitem{MaSiWe04}
\begin{barticle}[auto:STB|2012/07/16|09:13:59]
\bauthor{\bsnm{Martinelli},~\bfnm{F.}\binits{F.}},
  \bauthor{\bsnm{Sinclair},~\bfnm{Alistair~A.}\binits{A.~A.}} \AND
  \bauthor{\bsnm{Weitz},~\bfnm{D.}\binits{D.}}
(\byear{2004}).
\btitle{Glauber dynamics on trees: Boundary conditions and mixing time}.
\bjournal{Comm. Math. Phys.}
\bvolume{250}
\bpages{301--334}.
\bptok{imsref}%
\end{barticle}
\endbibitem

\bibitem{MezMon09}
\begin{bbook}[auto:STB|2012/07/16|09:13:59]
\bauthor{\bsnm{M{\'e}zard},~\bfnm{M.}\binits{M.}} \AND
  \bauthor{\bsnm{Montanari},~\bfnm{A.}\binits{A.}}
(\byear{2009}).
\btitle{Information, Physics, and Computation}.
\bpublisher{Oxford Univ. Press}, \baddress{Oxford}.
\bptok{imsref}%
\end{bbook}
\endbibitem

\bibitem{MoRiSe08}
\begin{barticle}[auto:STB|2012/07/16|09:13:59]
\bauthor{\bsnm{Montanari},~\bfnm{A.}\binits{A.}},
  \bauthor{\bsnm{Ricci-Tersenghi},~\bfnm{F.}\binits{F.}} \AND
  \bauthor{\bsnm{Semerjian},~\bfnm{G.}\binits{G.}}
(\byear{2008}).
\btitle{Clusters of solutions and replica symmetry breaking in random
  k-satisfiability}.
\bjournal{J. Stat. Mech. Theory Exp.}
\bvolume{2008}
\bpages{P04004}.
\bptok{imsref}%
\end{barticle}
\endbibitem

\bibitem{MosselSly08}
\begin{bincollection}[auto:STB|2012/07/16|09:13:59]
\bauthor{\bsnm{Mossel},~\bfnm{E.}\binits{E.}} \AND
  \bauthor{\bsnm{Sly},~\bfnm{A.}\binits{A.}}
(\byear{2008}).
\btitle{Rapid mixing of Gibbs sampling on graphs that are sparse on average}.
In \bbooktitle{Proceedings of the Nineteenth Annual ACM-SIAM Symposium on
  Discrete Algorithms (SODA)}
\bpages{238--247}.
\bptok{imsref}%
\end{bincollection}
\endbibitem

\bibitem{MosselSly09}
\begin{barticle}[auto:STB|2012/07/16|09:13:59]
\bauthor{\bsnm{Mossel},~\bfnm{E.}\binits{E.}} \AND
  \bauthor{\bsnm{Sly},~\bfnm{A.}\binits{A.}}
(\byear{2009}).
\btitle{Rapid mixing of Gibbs sampling on graphs that are sparse on average}.
\bjournal{Random Structures Algorithms}
\bvolume{35}
\bpages{250--270}.
\bptok{imsref}%
\end{barticle}
\endbibitem

\bibitem{MoWeWo09}
\begin{barticle}[auto:STB|2012/07/16|09:13:59]
\bauthor{\bsnm{Mossel},~\bfnm{E.}\binits{E.}},
  \bauthor{\bsnm{Weitz},~\bfnm{D.}\binits{D.}} \AND
  \bauthor{\bsnm{Wormald},~\bfnm{N.}\binits{N.}}
(\byear{2009}).
\btitle{On the hardness of sampling independent sets beyond the tree
  threshold}.
\bjournal{Probab. Theory Related Fields}
\bvolume{143}
\bpages{401--439}.
\bptok{imsref}%
\end{barticle}
\endbibitem

\bibitem{Peres05}
\begin{bmisc}[auto:STB|2012/07/16|09:13:59]
\bauthor{\bsnm{Peres},~\bfnm{Y.}\binits{Y.}}
\bhowpublished{Mixing for Markov chains and spin systems. Draft lecture notes}.
\bptok{imsref}%
\end{bmisc}
\endbibitem

\bibitem{StrookZegarlinski92}
\begin{barticle}[auto:STB|2012/07/16|09:13:59]
\bauthor{\bsnm{Stroock},~\bfnm{D.~W.}\binits{D.~W.}} \AND
  \bauthor{\bsnm{Zegarli{\'n}ski},~\bfnm{B.}\binits{B.}}
(\byear{1992}).
\btitle{The logarithmic Sobolev inequality for discrete spin systems on a
  lattice}.
\bjournal{Comm. Math. Phys.}
\bvolume{149}
\bpages{175--193}.
\bptok{imsref}%
\end{barticle}
\endbibitem

\bibitem{Vigoda99}
\begin{bincollection}[auto:STB|2012/07/16|09:13:59]
\bauthor{\bsnm{Vigoda},~\bfnm{E.}\binits{E.}}
(\byear{1999}).
\btitle{Improved bounds for sampling coloring}.
In \bbooktitle{40th Annual Symposium on Foundations of Computer Science (FOCS)}
\bpages{51--59}.
\bptok{imsref}%
\end{bincollection}
\endbibitem

\bibitem{Vigoda00}
\begin{barticle}[auto:STB|2012/07/16|09:13:59]
\bauthor{\bsnm{Vigoda},~\bfnm{E.}\binits{E.}}
(\byear{2000}).
\btitle{Improved bounds for sampling coloring}.
\bjournal{J. Math. Phys.}
\bvolume{3}
\bpages{1555--1569}.
\bptok{imsref}%
\end{barticle}
\endbibitem

\bibitem{Weitz05}
\begin{barticle}[auto:STB|2012/07/16|09:13:59]
\bauthor{\bsnm{Weitz},~\bfnm{Dror}\binits{D.}}
(\byear{2005}).
\btitle{Combinatorial criteria for uniqueness of Gibbs measures}.
\bjournal{Random Structures Algorithms}
\bvolume{27}
\bpages{445--475}.
\bptok{imsref}%
\end{barticle}
\endbibitem

\bibitem{Weitz06}
\begin{bmisc}[auto:STB|2012/07/16|09:13:59]
\bauthor{\bsnm{Weitz},~\bfnm{D.}\binits{D.}}
(\byear{2006}).
\bhowpublished{Counting indpendent sets up to the tree threshold. In
  \textit{Proceedings of the Thirty-eighth Annual ACM Symposium on Theory of
  Computing} 140--149. ACM, New York.}
\bptok{imsref}%
\end{bmisc}
\endbibitem

\bibitem{ZhaLiaBai11}
\begin{barticle}[mr]
\bauthor{\bsnm{Zhang},~\bfnm{Jinshan}\binits{J.}},
  \bauthor{\bsnm{Liang},~\bfnm{Heng}\binits{H.}} \AND
  \bauthor{\bsnm{Bai},~\bfnm{Fengshan}\binits{F.}}
(\byear{2011}).
\btitle{Approximating partition functions of the two-state spin system}.
\bjournal{Inform. Process. Lett.}
\bvolume{111}
\bpages{702--710}.
\bid{doi={10.1016/j.ipl.2011.04.012}, issn={0020-0190}, mr={2840539}}
\bptok{imsref}%
\end{barticle}
\endbibitem


\end{thebibliography}
\end{document}